\documentclass[10pt]{amsart}
\usepackage{latexsym,amssymb, amsmath, amsthm,xcolor,tikz-cd,hyperref,mathtools, geometry}
\hypersetup{
    colorlinks=true,
	allcolors=black,
	linkbordercolor={1 1 1}
}

\begin{document}
\newcommand\nvisom{\rotatebox[origin=cc] {-90}{$ \cong $}}
\newcommand{\Z}{\mathbb{Z}}
\newcommand{\Q}{\mathbb{Q}}
\newcommand{\F}{\mathbb{F}}
\newcommand{\kbar}{\overline{k}}
\newcommand{\pro}{\mathrm{Pro}}
\newcommand{\colim}{\mathrm{colim}}
\newcommand{\Fp}{\mathbb{F}_p}
\newcommand{\Hom}{\mathrm{Hom}}
\newcommand{\Spec}{\mathrm{Spec}}
\newcommand{\et}{\acute{e}t}
\newcommand{\Et}{\acute{E}t}
\newcommand{\Etr}{\acute{E}t^{\natural}_{/k}}
\newcommand{\xto}{\xrightarrow}
\newcommand{\Gal}{\mathrm{Gal}}
\newcommand{\nth}{n^{\text{th}}}

\newtheorem{thm}{Theorem}[section]
\newtheorem{cor}[thm]{Corollary}
\newtheorem{lemma}[thm]{Lemma}
\newtheorem{propn}[thm]{Proposition}
\newtheorem{con}[thm]{Conjecture}
\newtheorem{prop}[thm]{Property}
\theoremstyle{definition}
\newtheorem{example}[thm]{Example}
\newtheorem{defn}[thm]{Definition}
\newtheorem{rem}[thm]{Remark}
\newtheorem{qst}[thm]{Question}
\newtheorem{notation}[thm]{Notation}
\newtheorem{conj}[thm]{Conjecture}
\title[Power structures on the Grothendieck--Witt ring]{Power structures on the Grothendieck--Witt ring and the motivic Euler characteristic}
\author{Jesse Pajwani, Ambrus Pál}
\address{Department of Mathematical Sciences, University of Bath, North
Road, Bath, BA2 7AY, UK. \newline Department of Mathematics, 180 Queen's Gate, Imperial College, London, SW7 2AZ, United Kingdom}
\email{jp907@bath.ac.uk, a.pal@imperial.ac.uk}
\maketitle
\begin{abstract}
For $k$ a field, we construct a power structure on the Grothendieck--Witt ring of $k$ which has the potential to be compatible with symmetric powers of varieties and the motivic Euler characteristic. We then show our power structure is compatible with the variety power structure when we restrict to varieties of dimension $0$, using techniques of Garibaldi, Merkurjev and Serre about cohomological invariants.
\end{abstract}
\tableofcontents
\section{Introduction}
\subsection{Overview}

Let $k$ be a field of characteristic $\neq 2$. Recent work such as \cite{Le}, \cite{LPLS}, \cite{5author} and \cite{YZ} studies an arithmetic refinement of the Euler characteristic of a variety to an invariant which we refer to as the \emph{motivic Euler characteristic}.
\begin{defn}
Define $\widehat{\mathrm{W}}(k)$ to be the \emph{Grothendieck--Witt ring} of $k$, i.e. the ring generated by the symbols $[q]$ where $q$ is a non degenerate quadratic form over $k$ subject to the relations:
\begin{enumerate}
\item If $q \cong q'$, then $[q] = [q']$,
\item $[q] + [q'] = [q \oplus q']$,
\item $[q][q'] = [q \otimes q']$.
\end{enumerate}
For $\alpha \in k^\times$ let $\langle \alpha \rangle$ be the quadratic form on $k$ given by sending $x$ to $ax^2$. Then it is a well known fact, e.g. Theorem 4.1 of \cite{La}, that the symbols $\langle \alpha \rangle$ generate $\widehat{\mathrm{W}}(k)$. Write $\mathbb{H}$ for the hyperbolic form $\mathbb{H} := \langle 1 \rangle + \langle -1 \rangle$, and define the \emph{Witt ring} $\mathrm{W}(k) := \widehat{\mathrm{W}}(k)/ \langle \mathbb{H} \rangle$.

Let $K_0(\mathrm{Var}_k)$ be the \emph{Grothendieck ring of varieties} over $k$. That is, $K_0(\mathrm{Var}_k)$ is the ring generated by symbols $[X]$ where $X$ is a variety over $k$, subject to the following three relations.
\begin{enumerate}
\item If $X \cong X'$, then $[X] = [X']$,
\item if $U$ is an open subvariety of $X$, then $[X] = [U] + [X \setminus U]$,
\item if $X,Y$ are two varieties, then $[X] [Y] = [(X \times_k Y)_{\mathrm{red}}]$.
\end{enumerate}
Let $K_0(\Et_k)$ denote the subring of $K_0(\mathrm{Var}_k)$ which is generated by varieties of dimension $0$.

Let $X$ be a smooth projective variety over $k$. The cup product composed with the trace map gives rise to a quadratic form on the algebraic de Rham cohomology in even degrees, so we obtain a natural element of $\widehat{\mathrm{W}}(k)$, the \emph{de Rham Euler characteristic}:
$$
\chi^{dR}(X) := [ H^{2*}_{dR}(X/k) ] - \frac12 \cdot \mathrm{dim}_k ( H^{2*+1}_{dR}(X/k)) \cdot \mathbb{H} \in \widehat{\mathrm{W}}(k).
$$
Theorem 2.13 of \cite{5author} tells us that if $\mathrm{char}(k)=0$ the assignment $[X] \mapsto \chi^{dR}(X)$ on smooth projective varieties extends uniquely to a ring homomorphism
$$
\chi^{mot}: K_0(\mathrm{Var}_k) \to \widehat{\mathrm{W}}(k),
$$
which we call the \emph{motivic Euler characteristic}. If $A$ is a finite dimensional étale algebra over $k$, then $\chi^{mot}(\Spec(A)) = [\mathrm{Tr}_A]$, where $\mathrm{Tr}_A$ is the trace form. Therefore restricting $\chi^{mot}$ to $K_0(\Et_k)$ recovers the trace homomorphism, which also makes sense in positive characteristic. When the argument we are making only holds for varieties of dimension $0$, we will write $\mathrm{Tr}: K_0(\Et_k) \to \widehat{\mathrm{W}}(k)$ to make this clear. 
\end{defn}
The motivic Euler characteristic is a well studied ring homomorphism, and is the main object of study in the papers \cite{5author} and \cite{YZ}. It should be noted that $\chi^{mot}$ has an alternative definition in terms of motivic homotopy theory, which also allows us to define a version of $\chi^{mot}$ in positive characteristic (see Definition 2.7 of \cite{Az}). In \cite{5author}, the above definition is obtained from the motivic homotopy definition by applying Theorem 1.4 of \cite{LR} to compare the definition from motivic homotopy with Hodge cohomology, and then applying Theorem 3.1 of \cite{Bi} to show that this extends to a ring homomorphism. For the purpose of our paper, the above definition suffices, see Remark $\ref{positivecharproblems}$ for a further discussion about the positive characteristic case.

This paper is concerned with the compatibility of $\chi^{mot}$ with additional structures on the rings in question. Let $R$ be a commutative ring with unity. A power structure on $R$ is given by functions $\phi_n:  R \to R$ for all $n \in \mathbb{Z}_{\geq 0}$ satisfying certain axioms as in Corollary $\ref{powerfns}$. If $R,R'$ are rings with power structures $\phi_n, \phi'_n$, then a homomorphism $f: R \to R'$ respects the power structures if $f \circ \phi_n = \phi'_n \circ f$ for all $n$. If $k$ is a field of characteristic $0$, then for $R=K_0(\mathrm{Var}_k)$, symmetric powers of varieties give rise to a natural power structure, $S_n$, given by $S_i([X]) = [X^{(i)}]$. Moreover, this power structure restricts to give a power structure on $K_0(\Et_k)$, which we can also define in positive characteristic. This paper is concerned with the following power structure on $\widehat{\mathrm{W}}(k)$.
\begin{thm}[Corollary $\ref{finalpstruc}$]
Define $t_\alpha$ to be the element of $\widehat{\mathrm{W}}(k)$ given by $\langle 2 \rangle + \langle \alpha \rangle - \langle 1 \rangle - \langle 2\alpha\rangle$. Then there is a unique power structure on $\widehat{\mathrm{W}}(k)$ such that for any $\alpha \in k^\times$
$$
a_i(\langle \alpha \rangle) = \langle \alpha^i \rangle + \frac{i(i-1)}{2} (t_\alpha).
$$
\end{thm}
This power structure is not a standard power structure as discussed in Remark $\ref{surprising}$. Of particular note is the term $t_\alpha$: a $2$-torsion element in $\widehat{\mathrm{W}}(k)$. We study this term more closely in subsection $\ref{properties}$. In particular, Corollary $\ref{classicalpstruc}$ gives a cohomological condition on the mod $2$-Galois cohomology of $k$ which is equivalent to the vanishing of $t_\alpha$, and Corollary $\ref{2exoticpower}$ shows that this term vanishes for a number field $k$ if and only if $\sqrt{2} \in k$. 

The reason we introduce the above power structure is the first main theorem of this paper.
\begin{thm}[Corollary $\ref{maincor1}$]
Suppose $\mathrm{char}(k) = 0$ (resp. $\mathrm{char}(k) \neq 2$). Suppose that there exists a power structure $b_\bullet$ on $\widehat{\mathrm{W}}(k)$ such that $\chi^{mot}$ (resp. $\mathrm{Tr}$) respects the power structures on $K_0(\mathrm{Var}_k) \to \widehat{\mathrm{W}}(k)$ (resp. $K_0(\Et_k) \to \widehat{\mathrm{W}}(k)$. Then $b_\bullet=a_\bullet$.
\end{thm}
The proof of this theorem is split into two parts. Firstly, we show in Theorem $\ref{uniquepstruc}$ that any such $b_\bullet$ would be unique. We can then compute the value of $b_n(\langle \alpha \rangle)$ by looking at quadratic étale algebras, and do this in Corollary $\ref{quadcalc}$. However this doesn't imply that this gives a well defined power structure. This is resolved in Corollary $\ref{finalpstruc}$ where we use Theorem 4.1 of \cite{La} to show these functions are well defined and satisfy the necessary relations to give a well defined power structure. The computational need for the term $t_\alpha$ in the power structure is apparent from Corollary $\ref{ansdefn}$, and a potential explanation is discussed in Remark $\ref{Spheres}$.

It was pointed out to the authors by Kirsten Wickelgren the compatibility of these power structures has a link to a spectrum level enhancement of $\chi^{mot}$ coming from work of R{\"o}ndigs in \cite{R2} and Nanavaty in \cite{Na}. Determining whether a compatible power structure on $\widehat{\mathrm{W}}(k)$ exists may help to provide some answer to Question 1.14 of \cite{5author}, which asks what properties this spectrum level enhancement has. We would expect $\chi^{mot}$ to respect the power structures if the spectrum level map from \cite{Na} is a map of highly structured ring spectra. While we cannot show that $\chi^{mot}$ does respect the power structures above, we can when we restrict to $K_0(\Et_k)$, which is the second main theorem of this paper.
\begin{thm}[Corollary $\ref{maincor2}$]
Equip $K_0(\Et_k)$ and $\widehat{\mathrm{W}}(k)$ with the power structures above. Then the restriction of $\chi^{mot}$ to $K_0(\Et_k)$ respects the power structures on both rings.
\end{thm}
This theorem is proven in three parts. Firstly, we show by hand that the result is true for all quadratic and biquadratic étale algebras, i.e., if $A$ is a biquadratic étale algebra, then $\chi^{mot} ( S_n(\Spec(A))) = a_n(\chi^{mot}(\Spec(A))$. We then use a twisting argument from chapter 3 of \cite{YZ} in order to extend this result to all multiquadratic étale algebras. Finally, we appeal to an argument from \cite{GMS} in order to extend this to all étale algebras.

\subsection{Acknowledgements}
The authors would like to thank Stephen McKean for his useful discussions on the subject, and for initially raising the question of whether $\chi^{mot}$ respects the power structures to us. We would also like to thank Johannes Nicaise, who independently raised this with the authors, as well as the anonymous referee for this paper for additional context about the behaviour of $\chi^{mot}$ in characteristic $p$. The first author would also like to thank Kirsten Wickelgren for helpful discussions, as well as Kevin Buzzard for his careful proofreading of this work as part of the first author's thesis. This work was undertaken while the first author was supported by the London School of Geometry and Number Theory under the Engineering and Physical Sciences Research Council grant number [EP/S021590/1],  and he would also like to thank Imperial College London.
\section{Power structures on rings}
Following \cite{GZLMH}, we recall some definitions and properties of power structures on rings.
\begin{defn}\label{pstrucdef}
Let $R$ be a commutative ring. A \emph{power structure} on $R$ is given by a map of sets $(1+t\cdot R[[t]]) \times R \to 1+t\cdot R[[t]]$, which we write as $(f(t), r) \mapsto (f(t))^r$, such that:
\begin{enumerate}
\item $f(t)^0 = 1$,
\item $f(t)^1 = f(t)$,
\item $( f(t) \cdot g(t))^r = f(t)^r \cdot g(t)^r$,
\item $f(t)^{r+s} = f(t)^r \cdot f(t)^r$,
\item $f(t)^{rs} = (f(t)^r)^s$,
\item $(1+t)^m = 1+mt + $ terms of order $t^2$,
\item If $g(t) = f(t)^m$, then $f(t^i)^m = g(t^i)$ for any positive integer $i$.
\end{enumerate}
A power structure is \emph{finitely determined} if for any $N>0$, there exists $M >0$ such that if $f(t) \in 1+t\cdot R[[t]]$ and $m \in R$, we may determine $f(t)^m \pmod{t^N}$ solely from $m$ and $f(t) \pmod{t^M}$. For the rest of this paper, all power structures will be finitely determined, and we will simply refer to ``power structures" to mean ``finitely determined power structures".

Let $S_R := \{ f \in 1 + t \cdot R[[t]]: f(t) = (1-t^i) \text{ for some i}\}$. A \emph{pre-power structure} on $R$ is a map $S_R \times R \to 1+t\cdot R[[t]]$ given by $((1-t^i), r) \mapsto (1-t^i)^r$ such that:
\begin{enumerate}
\item $(1-t)^{-r} = 1+rt+ o(t^2)$,
\item $(1-t)^{-(r+s)} = (1-t)^{-r} \cdot (1-t)^{-s}$,
\item $(1-t)^{-1} = \sum_{i=0}^{\infty} t^i$,
\item If $f(t) = (1-t)^a$, then $(1-t^i)^a = f(t^i)$. 
\end{enumerate}
Note that property $4$ of pre-power structures tells us that it is sufficient to determine a pre-power structure simply by looking at $(1-t)^{-a}$.
\end{defn}

\begin{lemma}
Suppose that there exists a power structure on $R$. Then restricting this power structure to $S_R$ gives rise to a pre-power structure.
\end{lemma}
\begin{proof}
Suppose we have a power structure on $R$. Properties $2$ and $4$ of pre-power structures are trivially satisfied by properties of power structures. Note that
$$
(1-t)^{-1} (1-t) = (1-t)^{-1} (1-t)^1 = (1-t)^0 = 1,
$$
which tells us that $(1-t)^{-1} = (1+t+t^2 + \ldots)$, as required. 

It only remains to show that $(1-t)^{-a} = 1+at + o(t^2)$. By property $3$ of power structures, $(1-t)^{-a}(1+t)^{-a} = (1-t^2)^{-a}$. By property $7$, the coefficient of $t$ in $(1-t^2)^{-a}$ is $0$, and so we obtain $(1-t^2)^{-a} = 1 + o(t^2)$. By property $6$, $(1+t)^{-a}  = 1 - at + o(t^2)$. Therefore since $(1-t)^{-a} ( 1 -  at + o(t^2)) = 1 + o(t^2)$, we see that $(1-t)^{-a} = 1 + at + o(t^2)$, as required.
\end{proof}
\begin{propn}[Proposition 1 of \cite{GZLMH}]\label{propn1}
Let $R$ be a ring with a pre-power structure. Then the pre-power structure on $R$ extends uniquely to a power structure on $R$.
\end{propn}

\begin{cor}\label{unique}
A power structure on $R$ is uniquely determined by the elements $(1-t)^{-r}$ for all $r \in R$.
\end{cor}
\begin{proof}
These elements uniquely determine a pre-power structure on $R$, so this follows from the proposition above. 
\end{proof}
\begin{cor}\label{powerfns}
A power structure on $R$ gives rise to natural functions $a_i: R \to R$ for all non-negative integers $i$ such that for all $a_0(r)=1$ for all $r$ and $a_1 = \mathrm{Id}$. Moreover, for $i\geq1$ we have $a_i(0)=0$ and $a_i(1) = 1$. Finally for any $r,s \in R$, we have  $a_n(r+s) =  \sum_{i=0}^{n-1} a_i(r)a_i(s)$. Conversely, functions satisfying these axioms uniquely determine a power structure.
\end{cor}
\begin{proof}
Define $a_0$ to be identically $1$, and define $a_i(r)$ to be such that
$$
(1-t)^{-a} =  1 + \sum_{i=1}^{\infty} a_i(r)t^i.
$$
Clearly the functions $a_i$ uniquely determine the power structure, by the previous corollary, so we only need to show they have the properties specified in the statement of the corollary. By definition, $a_1=\mathrm{Id}$. We also see $a_i(0)=0$ and $a_i(1)=1$ for all $i$. Finally, property 4 of power structures tells us that the $\alpha_i$ must satisfy the required equality. For the reverse implication, these $a_i$ uniquely determine a pre-power structure on $R$, and so determine a power structure on $R$ as well.
\end{proof}
\begin{defn}
Let $R, R'$ be rings, and let $\varphi: R \to R'$ be a homomorphism.  Abusing notation slightly, let $\varphi$ denote the map $\varphi: R[[t]] \to R'[[t]]$, given by $\varphi(\sum_{i=0}^{\infty} \alpha_i t^i) =\sum_{i=0}^{\infty} \varphi(\alpha_i) t^i.$ Suppose that $R, R'$ both have power structures on them. We say $\varphi$ \emph{respects the power structures on $R$ and $R'$} if $\varphi(f(t)^a) = \varphi(f(t))^{\varphi(a)}$ for all $f(t) \in 1+t\cdot R[[t]]$ and all $a \in R$.
\end{defn}
\begin{propn}[Proposition 2 of \cite{GZLMH}]
Let $R, R'$ be rings with power structures. Let $\varphi: R \to R'$ be a ring homomorphism such that $(1-t)^{-\varphi(a)} = \varphi( (1-t)^{-a})$.  Then $\varphi$ respects the power structures on $R$ and $R'$.
\end{propn}

\begin{cor}\label{surj}
Let $R, R'$ be rings and suppose there is a power structure on $R$. Let $\varphi: R \to R'$ be a surjective ring homomorphism. Then if there exists a power structure on $R'$ such that $\varphi$ respects the power structures, this power structure is unique.
\end{cor}
\begin{proof}
Suppose there exists a power structure on $R'$ such that $\varphi$ respects the power structures. Let $g(t) \in R'[[t]]$, and let $f(t) \in R[[t]]$ such that $\varphi(f(t)) = g(t)$, which exists by surjectivity of $\varphi$. Similarly, let $\beta \in R'$, and $\alpha \in R$ with $\varphi(\alpha) = \beta$. Since $\varphi$ respects the power structures, we must have $g(t)^\beta = \varphi(f(t))^{\varphi(\alpha)} = \varphi( f(t)^\alpha)$, since $\varphi$ respects the power structures. Therefore, the power structure on $R'$ is uniquely determined by the power structure on $R$.
\end{proof}

Let $R, R'$ be rings with power structures on them. Let $a_i: R \to R$ be the collection of functions defining the power structure as in Corollary $\ref{powerfns}$, and let $b_i: R' \to R'$ be the corresponding functions for $R'$.  Let $\varphi: R \to R'$ be a map of rings which, a priori, has no relationship to the power structures on $R, R'$. Note that $\varphi$ is compatible with the power structures on $R, R'$ if and only if $\varphi( a_n(r)) = b_n(\varphi(r))$ for all $r \in R$ and for all $n \geq 0$. 

\begin{lemma}\label{subgroup}
Suppose that $r, s \in R$ are elements such that for all $n$, $\varphi(a_n(r)) = b_n(\varphi(r))$ and similarly for $s$. Then the same is true for $-r$ and $r+s$. In particular, the set of elements of $R$ such that $\varphi$ is compatible with the power structure operations on them is a subgroup of $R$.
\end{lemma}
\begin{proof}
Note that $a_n(r+s) = \sum_{i=0}^n a_i(r)a_{n-i}(s)$. Applying $\varphi$ to both sides of this and using the induction hypothesis gives the result for $r+s$.  We now claim that $\varphi(a_n(-r)) = b_n(\varphi(-r))$. The result clearly holds for $n=1$, so suppose it holds for all $n'<n$ for induction. Since we have that $0 = a_n(0) = a_n(r - r ) = \sum_{i=0}^n a_i(r)a_{n-i}(-r)$, we can write:
$$
a_n(-r) = - \sum_{i=0}^{n-1} a_i(r)a_{n-i}(-r),
$$ 
and similarly for $b_n(-r)$.  Applying $\varphi$ to both sides of the above equation gives us:
$$
\varphi(a_n(-r)) = - \sum_{i=0}^{n-1} \varphi(a_i(r)) \varphi(a_{n-i}(-r)) = - \sum_{i=0}^{n-1} b_i(\varphi(r))b_{n-i}(\varphi(-r)) = b_n(\varphi(-r)),
$$
as required. 
\end{proof}

\section{Determining a possible power structure on $\widehat{\mathrm{W}}(k)$}
This section is dedicated to proving Corollary $\ref{maincor1}$. We first proceed by showing that any power structure on $\widehat{\mathrm{W}}(k)$ satisfying the conditions of the theorems would be unique. This then allows us to compute exactly what it must be, by specialising to quadratic étale algebras. Finally, we show that these functions are well defined and give rise to a genuine power structure, before showing some basic properties that this power structure satisfies. 

\subsection{Uniqueness of the power structure}
\begin{defn}
 For $X$ a variety over $k$, let $X^{(m)}$ denote the \emph{$m^{\text{th}}$ symmetric power} of $X$, ie, the quotient of $X^m$ by the $S_m$ action permuting the copies of $X$ in the product.
\end{defn}
\begin{thm}[\cite{GZLMH}]\label{powerstructureonK0vark}
For $\mathrm{char}(k)=0$, there is a power structure on $K_0(\mathrm{Var}_k)$ given on quasiprojective varieties by
$$
(1-t)^{-[X]} = 1 + [X]t + [X^{(2)}]t^2 + \ldots.
$$
\end{thm}
\begin{defn}
Note, for $X/k$ any quasiprojective variety, we have that the coefficient of $t^n$ of $(1-t)^{-[X]}$ is precisely $X^{(n)}$. Therefore, for $Y$ any element of $K_0(\mathrm{Var}_k)$, write $Y^{(n)}$ to mean the coefficient of $t^n$ in the power series $(1-t)^{-[Y]}$. Note that if $X$ is a variety of dimension $0$, then so is $X^{(n)}$, so the power structure above restricts to give a power structure on $K_0(\Et_k)$ as well.
\end{defn}
In the language of Corollary $\ref{powerfns}$, this is saying that there is a power structure $K_0(\mathrm{Var}_k)$ determined by the functions $S_n([X]) = [X^{(n)}]$.

 It is unclear whether a similar statement holds in characteristic $p$. In characteristic $0$, Bittner's theorem (Theorem 3.1 of \cite{Bi}) guarantees that $K_0(\mathrm{Var}_k)$ is generated by quasi-projective varieties, which in turn implies that defining the power structure on quasi-projective varieties is sufficient. However, when restricting to varieties of dimension $0$, we have the following.
\begin{lemma}
Let $k$ be a field of characteristic $\neq 2$. Then there is a power structure on $K_0(\Et_k)$ given on the level of finite étale algebras $A$ by
$$
(1-t)^{-[A]} = 1 + [A]t + [A^{(2)}]t^2 + \ldots .
$$
\end{lemma}
\begin{proof}
This follows by identical reasoning to characteristic $0$, noting that all dimension $0$ varieties are quasi-projective, so $K_0(\Et_k)$ is generated by quasi-projective varieties.
\end{proof}
\begin{rem}\label{positivecharproblems}
Since this paper is concerned with studying a power structure on $\widehat{\mathrm{W}}(k)$ coming from $\chi^{mot}$ and symmetric powers, we will restrict either to studying the following two cases
\begin{enumerate}
\item Studying $\chi^{mot}: K_0(\mathrm{Var}_k) \to \widehat{\mathrm{W}}(k)$ when $\mathrm{char}(k)=0$.
\item Studying $\mathrm{Tr}: K_0(\Et_k) \to\widehat{\mathrm{W}}(k)$ when $\mathrm{char}(k)\neq 2$.
\end{enumerate}
It it is possible to define $\chi^{mot}$ for a general variety in characteristic $p$: over a perfect field, this is Definition 2.7 of \cite{Az}, and as pointed out by the referee of this paper this definition extends to imperfect fields by modifying Proposition 2.6 of \cite{Az}, so that instead of using the duality statement of Riou (\cite{LYZR}, Corollary B2), we use Theorem 3.2.1 of \cite{EK}. 

For the purpose of this paper it suffices to define $\chi^{mot}$ for varieties of dimension $0$ in positive characteristic, as these are the varieties where we have this power structure. These are all projective, and so it suffices to use de Rham cohomology as in the introduction.
\end{rem}

\begin{lemma}\label{surjK0}
For $\mathrm{char}(k) \neq 2$, the trace map $\mathrm{Tr}: K_0(\Et_k) \to \widehat{\mathrm{W}}(k)$ is surjective.
\end{lemma}
\begin{proof}
It is enough to show that for any $\alpha \in (k^\times)/(k^\times)^2$, there exists some $X_\alpha \in K_0(\Et_k)$ such that $\mathrm{Tr}(X_\alpha) = \langle \alpha \rangle$.  For $\alpha \in (k^\times)^2$, this is easy: let $X_\alpha := [k]$. 

Firstly, suppose $2 \in (k^\times)^2$. Then $\mathrm{Tr}(k(\sqrt{\alpha})) = \langle 2 \rangle + \langle 2\alpha \rangle = \langle 1 \rangle + \langle \alpha \rangle$. This means if we define $X_\alpha := [k(\sqrt{\alpha})] - [k]$ for $\alpha \not\in (k^\times)^2$, this gives the result. 

Now suppose that $2 \not \in (k^\times)^2$. Then $\mathrm{Tr}(k(\sqrt{2})) = \langle 2 \rangle + \langle 1 \rangle$. For $\alpha \in (k^\times)/(k^\times)^2$,  define $X_\alpha := [k(\sqrt{2\alpha})] - [k(\sqrt{2})] + [k]$. Then $\mathrm{Tr}(X_\alpha) = \langle \alpha \rangle$ as required. 
\end{proof}
We deduce the following corollary, which was also mentioned in Theorem 1 of \cite{R2}.
\begin{cor}
The motivic Euler characteristic $\chi^{mot}: K_0(\mathrm{Var}_k) \to \widehat{\mathrm{W}}(k)$ is surjective.
\end{cor}
\begin{proof}
Since $\mathrm{Tr} = \chi^{mot}|_{K_0(\Et_k)}$, this is immediate.
\end{proof}

\begin{cor}\label{uniquepstruc}
Suppose $\mathrm{char}(k) = 0$ (resp. $\neq 2$). If there exists a power structure on $\widehat{\mathrm{W}}(k)$ such that $\chi^{mot}: K_0(\mathrm{Var}_k) \to \widehat{\mathrm{W}}(k)$ (resp. $\mathrm{Tr}: K_0(\Et_k) \to \widehat{\mathrm{W}}(k)$) respects the power structures, then the power structure on $\widehat{\mathrm{W}}(k)$ is unique.
\end{cor}
\begin{proof}
When $\mathrm{char}(k)=0$, since $\chi^{mot}$ is surjective, this follows by Corollary $\ref{surj}$. The statement in positive characteristic holds since $\mathrm{Tr}$ is surjective. 
\end{proof}
\begin{rem}\label{questionoverC}
If $k =\mathbb{C}$, we have a canonical isomorphism $\widehat{\mathrm{W}}(\mathbb{C}) = \Z$, and Remark 1.3 of \cite{Le} allows us to identify $\chi^{mot}$ with the function sending a variety $X$ to the compactly supported Euler characteristic of $X(\mathbb{C})$. We have a canonical power structure on $\Z$, given by $a_n(m) = {n+m -1 \choose n}$. The main theorem of \cite{Mac} tells us that $\chi^{mot}$ respects the power structures on both sides. This raises the question of whether there exists a power structure on $\widehat{\mathrm{W}}(k)$ such that $\chi^{mot}$ respects the power structures for more general $k$, or whether there is an arithmetic obstruction to this compatibility.
\end{rem}

\subsection{Determining the power structure}

In this subsection we use results about motivic Euler characteristics of symmetric products of quadratic étale algebras to determine what the unique power structure specified by Corollary $\ref{uniquepstruc}$ must be if it exists. This allows us to give an explicit computation of the potential power structure on $\widehat{\mathrm{W}}(k)$. 

\begin{defn}\label{pa}
Let $\alpha \in k^{\times}$. Let $k[\sqrt{\alpha}] := k[x]/(x^2-\alpha)$, and define $P_\alpha := \Spec(k[\sqrt{\alpha}])$. We can easily compute $\chi^{mot}(P_\alpha) = 2\langle 1 \rangle$ for $\alpha \in (k^\times)^2$, and $\langle 2 \rangle + \langle 2\alpha \rangle$ otherwise. 
\end{defn}
\begin{lemma}\label{quadratic} We can compute the symmetric powers of $P_\alpha$ to be
\begin{equation*}
P_\alpha^{(n)}=
\begin{cases}
\amalg_{i=1}^{m+1} P_\alpha & \text{if $n=2m+1$ is odd,}\\
\Spec(k) \amalg \coprod_{i=1}^m P_\alpha, & \text{if $n=2m$ is even.}
\end{cases}
\end{equation*}
\end{lemma}
\begin{proof} Assume that $k(\sqrt a)$ is a degree two field extension of $k$, otherwise the claim is trivial.

We can identify $P_\alpha^n(\kbar)$ with the set of $n$-tuples $(\alpha_1, \ldots, \alpha_n)$, where each $\alpha_i \in \{-1,1\}$. This is because an object of $P_\alpha^n(\overline k)$ correponds to an $n$-tuple of embeddings of $k(\sqrt{\alpha})$ into $\kbar$, and since there are two embeddings, identify one of them with $-1$ and the other with $1$. Under this identification the action of $\Gal_k$ factors through through $\Gal(k(\sqrt{\alpha})/k)$, and the non trivial element of $\Gal(k(\sqrt{\alpha})/k)$ acts by sending $(\alpha_1, \ldots, \alpha_n)$ to $(-\alpha_1, \ldots, -\alpha_n)$.  This action commutes with the natural action of $S_n$ permuting the order of the sequences, so it induces an action of $\Gal_k$ on the $S_n$-orbits. The latter is in a $\Gal_k$-equivariant bijection with $P_\alpha^{(n)}(\overline k)$. The $S_n$-orbits are sequences with exactly $i$ zeros for each $i=0,1,\ldots,n$, and the involution of $\mathrm{Gal}(k(\sqrt{\alpha})/k)$ switches the orbit of $i$ zeros with the the orbit of $n-i$ zeros. Therefore every orbit of this involution consists of two elements except when $n=2m$ is even, since then  the orbit of $m$ zeros is fixed. This means there are isomorphisms of finite $\Gal_k$-sets:
\begin{equation*}
P_\alpha^{(n)}(\kbar) \cong
\begin{cases}
\amalg_{i=1}^{m+1} P_\alpha(\kbar) & \text{if $n=2m+1$ is odd,}\\
\Spec(k)(\kbar) \amalg \coprod_{i=1}^m P_\alpha(\kbar), & \text{if $n=2m$ is even.}
\end{cases}
\end{equation*}
Since $P_\alpha^{(n)}$ is also a $0$-dimensional variety over $k$, this gives the result. 
\end{proof}
\begin{cor}\label{quadcalc}
We can compute 
\begin{align*}
\chi^{mot}(P_\alpha^{(n)}) &= \frac{n+1}{2} \left( \langle 2 \rangle + \langle 2\alpha\rangle \right)\text{ if $n$ is odd,}\\
\chi^{mot}(P_\alpha^{(n)}) &= \langle 1 \rangle +  \frac{n}{2}  \left( \langle 2 \rangle + \langle 2\alpha\rangle \right) \text{if $n$ is even.}
\end{align*}
\end{cor}
\begin{proof}
Apply $\chi^{mot}$ to the lemma above.
\end{proof}

For the rest of this subsection, if $\mathrm{char}(k)=0$, assume there exists a power structure on $\widehat{\mathrm{W}}(k)$ that is compatible with $\chi^{mot}$, i.e., functions $a_n: \widehat{\mathrm{W}}(k) \to \widehat{\mathrm{W}}(k)$ such that $\chi^{mot}(X^{(n)}) = a_n( \chi^{mot}(X))$ satisfying the axioms from Corollary $\ref{powerfns}$. If $\mathrm{char}(k)$ is positive, suppose that $\mathrm{char}(k) \neq 2$ and that the same is true if we replace $K_0(\mathrm{Var}_k)$ with $K_0(\Et_k)$ and $\chi^{mot}$ with $\mathrm{Tr}$. 

We show this assumption and Lemma $\ref{quadratic}$ determines the $a_n$ functions uniquely.
\begin{lemma}\label{2power}
We can compute $a_n(\langle 2 \rangle) = \langle 2^n \rangle$. 
\end{lemma}
\begin{proof}
Proceed by induction on $n$. The cases $n=0,1$ are clear, so assume that the claim holds for all $m<n$. For any $b\in k^\times$, $\langle b^n \rangle = \langle \beta \rangle$ if $n$ is odd and $\langle 1 \rangle$ if $n$ is even. Therefore
$$\chi^{mot}(P_{1/2}^{(n)})=\langle1\rangle+\langle2\rangle+\langle2^2\rangle+\cdots+\langle 2^n
\rangle$$
by Lemma \ref{quadratic}. Using now that $a_n(\chi^{mot}(P_{1/2})) = \chi^{mot}(P_{1/2}^{(n)})$, we obtain
\begin{align*}
\langle1\rangle+\langle2\rangle+\langle2^2\rangle+\cdots+\langle 2^n\rangle &= a_n(\langle1\rangle+\langle2\rangle)=\sum_{i=0}^na_i(\langle2\rangle)
a_{n-i}(\langle1\rangle)\\
&=
\langle1\rangle+\langle2\rangle+\langle2^2\rangle+\cdots+\langle 2^{n-1}\rangle+a_n(\langle2\rangle)
\end{align*}
by the induction hypothesis. The claim is now clear. 
\end{proof}
\begin{cor}\label{ansdefn}
We can compute $a_n(\langle \alpha \rangle) = \langle \alpha^n \rangle + \lfloor \frac{n}{2} \rfloor \left( \langle 2 \rangle + \langle \alpha \rangle  - \langle 1 \rangle - \langle 2\alpha \rangle\right)$.
\end{cor}
\begin{proof}
Note that $a_n(\langle 2 \rangle + \langle \alpha \rangle) = \chi^{mot}(P^{(n)}_{\frac{\alpha}{2}})$, since the power structures are compatible. However, we also see that
$$
a_{n+1} (\langle 2 \rangle + \langle \alpha \rangle) = \sum_{i=0}^{n+1} a_i(\langle \alpha \rangle) \langle 2^{n+1-i} \rangle,
$$
using the definition of $a_i$ and Lemma $\ref{2power}$. Therefore we see that 
$$
\langle 2 \rangle \left( \sum_{i=0}^n a_i(\langle \alpha \rangle) \langle 2^{n-i} \rangle \right) + a_{n+1}(\langle \alpha \rangle) = a_{n+1}(\langle 2 \rangle + \langle \alpha \rangle).
$$
Rearranging this and using that $a_n(\langle 2 \rangle + \langle \alpha \rangle) = \chi^{mot}(P^{(n)}_{\frac{\alpha}{2}})$ tells us
$$
a_{n+1}(\langle \alpha \rangle) = \chi^{mot}( P^{(n+1)}_{\frac{\alpha}{2}}) - \langle 2 \rangle \chi^{mot}(P^{(n)}_{\frac{\alpha}{2}}).
$$
Suppose now $n+1$ is even. Then we can compute $\chi^{mot}( P^{(n+1)}_{\frac{\alpha}{2}}) = \langle 1 \rangle + \frac{n+1}{2}( \langle 2 \rangle + \langle \alpha \rangle)$, and we can also compute $\chi^{mot}(P^{(n)}_{\frac{\alpha}{2}}) = \frac{n+1}{2}(\langle 2 \rangle + \langle \alpha \rangle)$. We therefore get
$$
a_{n+1}( \langle \alpha \rangle) = \langle 1 \rangle + \frac{n+1}{2}( \langle 2 \rangle + \langle \alpha \rangle)( \langle 1 \rangle - \langle 2 \rangle).
$$
Expanding $( \langle 2 \rangle + \langle \alpha \rangle)( \langle 1 \rangle - \langle 2 \rangle) = \langle 2 \rangle + \langle \alpha \rangle - \langle 1 \rangle - \langle 2\alpha \rangle$ gives the result.

Suppose $n+1$ is odd instead. Then we can compute $\chi^{mot}( P^{(n+1)}_{\frac{\alpha}{2}}) = \frac{n+2}2(\langle 2 \rangle + \langle \alpha \rangle)$ and similarly, $\chi^{mot}(P^{(n)}_{\frac{\alpha}{2}}) = \langle 1 \rangle + \frac{n}{2}( \langle 2 \rangle + \langle \alpha \rangle)$. We therefore obtain
\begin{align*}
a_{n+1}(\langle \alpha \rangle)  &=  \frac{n+2}2(\langle 2 \rangle + \langle \alpha \rangle) - \langle 2 \rangle - \frac{n}{2} \langle 2 \rangle \cdot (\langle 2 \rangle + \langle \alpha \rangle) \\
&= \langle \alpha \rangle + \frac{n}{2} ( \langle 2 \rangle + \langle \alpha \rangle) (\langle 1 \rangle - \langle 2 \rangle),
\end{align*}
and expanding the brackets gives the result. 
\end{proof}

\begin{rem}\label{welldefinedconcerns}
Corollary $\ref{ansdefn}$ does not guarantee that these $a_n$ functions give well defined maps $\widehat{\mathrm{W}}(k) \to \widehat{\mathrm{W}}(k)$. We can define $a_n(q)$ for any $q$ by writing $q = \sum_{i=1}^m \langle \alpha_i \rangle$, applying $a_n$, and using the sum formula that the $a_n$ must satisfy. However, it is not clear that the value of $a_n(q)$ is independent of the presentation of $q$. This is addressed in Corollary $\ref{finalpstruc}$. 
\end{rem}

\subsection{The power structure is well defined}
In this subsection we show the concerns raised in the remark above are not a problem. That is, we have well defined functions $a_n: \widehat{\mathrm{W}}(k) \to \widehat{\mathrm{W}}(k)$ agreeing with the computation from Corollary $\ref{ansdefn}$ that give rise to a power structure on $\widehat{\mathrm{W}}(k)$.
\begin{defn}\label{zkdefn}
Consider the group ring $\Z[k^\times]$, where we write $\langle \alpha \rangle$ to mean the element of $\Z[k^\times]$ corresponding to $\alpha$, for any $\alpha \in k^\times$. Let $R$ denote the subgroup of $\Z[k^\times]$ generated by elements of the form $\langle \alpha \rangle - \langle \alpha \beta^2 \rangle$ or of the form $(\langle \alpha \rangle + \langle \beta \rangle) - (\langle \alpha + \beta \rangle + \langle (\alpha+\beta)\alpha \beta \rangle)$, for any $\alpha,\beta \in k^\times$. By Theorem 4.1 of \cite{La}, we have that $\Z[k^\times]/R \cong \widehat{\mathrm{W}}(k)$.

Define $t_\alpha :=  ( \langle 2 \rangle + \langle \alpha \rangle - \langle 1 \rangle - \langle 2\alpha \rangle) \in \widehat{\mathrm{W}}(k)$. Define the functions $a_n: \Z[k^\times] \to \widehat{\mathrm{W}}(k)$ for all positive integers $n$ by
\begin{enumerate}
\item $a_1( q) = [q] \in \widehat{\mathrm{W}}(k) \cong \Z[k^\times]/R$,
\item $a_n(q + q') = a_n(q) + a_n(q') + \sum_{i=1}^{n-1} a_i(q)a_{n-i}(q')$,
\item $a_n(\langle \alpha \rangle) = \langle \alpha^n \rangle + \lfloor \frac{n}{2} \rfloor t_\alpha$.
\end{enumerate}
\end{defn}
This definition allows us to talk about these functions $a_n$ as well defined functions from $\Z[k^\times]$. We would like to show that these descend to functions $\widehat{\mathrm{W}}(k) \to \widehat{\mathrm{W}}(k)$. We first need some basic facts about the elements $t_\alpha \in \widehat{\mathrm{W}}(k)$.
\begin{lemma}
Let $\alpha \in k^\times$. Then $2 \langle \alpha \rangle = 2 \langle 2\alpha \rangle \in \widehat{\mathrm{W}}(k)$. 
\end{lemma}
\begin{proof}
The second relation defining $R$, with $a=b$, gives $2 \langle \alpha \rangle = \langle 2\alpha \rangle + \langle 2\alpha^3 \rangle \in \widehat{\mathrm{W}}(k)$. We then apply the first relation defining $R$ to obtain $\langle 2\alpha^3 \rangle = \langle 2\alpha \rangle$, and the result is now clear.
\end{proof}
\begin{cor}\label{twotimesta}
The element $t_\alpha = ( \langle 2 \rangle + \langle \alpha \rangle - \langle 1 \rangle - \langle 2\alpha \rangle) \in \widehat{\mathrm{W}}(k)$ is $2$-torsion, and therefore satisfies $\langle 2 \rangle t_\alpha = t_\alpha$ for any $a$.
\end{cor}
\begin{proof}
The first part is immediate by the above relation. The second part follows since it is clear that $\langle 2 \rangle t_\alpha = -t_\alpha$, so since $t_\alpha$ is $2$-torsion this gives the result.
\end{proof}

\begin{cor}
Let $a_n$ be as in Definition $\ref{zkdefn}$. Then 
\begin{equation*}
a_n(\langle \alpha \rangle) =
\begin{cases}
 \langle \alpha^n \rangle &\text{ for $n \equiv 0,1 \pmod{4}$},\\
 \langle \alpha^n \rangle +  t_\alpha &\text{ for $n \equiv 2,3 \pmod{4}$. }
\end{cases}
\end{equation*}
\end{cor}
\begin{proof}
This depends on whether $\lfloor \frac{n}{2} \rfloor$ is even or odd, which is the congruence condition stated. 
\end{proof}

\begin{lemma}\label{prodts}
The element $t_\alpha$ satisfies $t_\alpha t_\beta=0$ for any $\alpha, \beta \in k^\times/(k^\times)^2$.
\end{lemma}
\begin{proof}
Note that $t_\alpha = (\langle 2 \rangle + \langle \alpha \rangle)(\langle 1 \rangle - \langle 2 \rangle)$. Since $(\langle 1 \rangle - \langle 2 \rangle)^2 = 2\langle 1 \rangle - 2 \langle 2 \rangle = 0$, we obtain the result.
\end{proof}

\begin{lemma}\label{tab}
For $a,b \in k^\times$, we have $\langle \alpha \rangle t_\beta - t_{\alpha \beta} = t_\alpha$.
\end{lemma}
\begin{proof}
This is an easy direct computation:
$$
\langle \alpha \rangle t_\beta - t_{\alpha \beta} = (\langle 1 \rangle - \langle 2 \rangle) (\langle 2\alpha \rangle + \langle \alpha \beta \rangle - \langle 2 \rangle - \langle \alpha \beta \rangle) = t_{2\alpha} = t_\alpha.
$$
\end{proof}
\begin{cor}\label{ctab}
For any $\alpha,\beta,\gamma \in k^\times$, $\langle \gamma \rangle (t_\alpha + t_\beta) = t_{\alpha \gamma} + t_{\beta \gamma}$. In particular we can compute $\langle \alpha \rangle t_\alpha = t_\alpha$ and $\langle \alpha \rangle t_\beta + \langle \beta \rangle t_\alpha = t_\beta + t_\alpha$.
\end{cor}
\begin{proof}
Use the above lemma to write $t_\alpha+t_\beta = \langle \alpha \rangle t_{\alpha \beta}$, then use the above lemma again. The second part is an application of the first part, using that $t_{\alpha^2}=0$ and $t_{\alpha \beta}$ is $2$-torsion.
\end{proof}

\begin{lemma}\label{sumts}
For all $\alpha,\beta$ and for all $n$, the elements $t_\alpha$ satisfy $t_\beta +  t_\alpha = t_{(\alpha+\beta)\alpha \beta} +t_{\alpha+\beta}.$
\end{lemma}
\begin{proof}
For any $\alpha,\beta$, we compute
$$
t_\alpha + t_\beta = (\langle 2 \rangle - \langle 1 \rangle)( \langle 2 \rangle + \langle 2 \rangle + \langle \alpha \rangle + \langle \beta \rangle),
$$
so the lemma holds using the relation $\langle \alpha \rangle + \langle \beta \rangle = \langle \alpha+\beta \rangle + \langle (\alpha+\beta)\alpha \beta \rangle$ in $\widehat{\mathrm{W}}(k)$.
\end{proof}

\begin{lemma}\label{compute}
We can compute
\begin{equation*}
a_n(\langle \alpha \rangle + \langle \beta \rangle) = 
\begin{cases}
(1+\frac{n}2)\langle 1 \rangle + \frac{n}{2}(\langle \alpha \beta \rangle + t_\alpha + t_\beta) & n \text{ even}, \\
\frac{n+1}{2}( \langle \alpha \rangle + \langle \beta \rangle)& n \text{ odd}.\\
\end{cases}
\end{equation*}
\end{lemma}
\begin{proof}
We show this by direct computation, splitting this into $4$ cases depending on $n \pmod 4$. The proof is a straightforward direct computation in each case, and we work through the case $n \equiv 0 \pmod{4}$, leaving the other cases to the reader. We compute
\begin{align*}
a_n(\langle \alpha \rangle + \langle \beta \rangle) &= a_n(\langle \alpha \rangle) + a_n(\langle \beta \rangle) + \sum_{i=1}^{n-1} a_i( \langle \alpha \rangle)a_{n-i}(\langle \beta \rangle)\\
 &= \langle \alpha^n \rangle + \langle \beta^n \rangle + \sum_{i=1}^{n-1} a_i( \langle \alpha \rangle)a_{n-i}(\langle \beta \rangle).
\end{align*}
Suppose $n \equiv 0 \pmod{4}$. Expanding $a_i(\langle \alpha \rangle)$ and $a_{n-i}(\langle \beta \rangle)$ and using Lemma $\ref{prodts}$, we obtain:
\begin{equation*}
a_i(\langle \alpha \rangle)a_{n-i}(\langle \beta \rangle)=
\begin{cases}
\langle 1 \rangle &\text{ if $i \equiv 0 \pmod{4}$,} \\
\langle \alpha \beta \rangle + \langle \alpha \rangle t_\beta  &\text{ if $i \equiv 1 \pmod{4}$,} \\
\langle 1 \rangle +  t_\alpha + t_\beta &\text{ if $i \equiv 2 \pmod{4}$,}\\
\langle \alpha \beta \rangle + \langle \beta \rangle t_\alpha &\text{ if $i \equiv 3 \pmod{4}$}.
\end{cases}
\end{equation*}
Let $m=\frac{n}4$. Then each term above appears $m$ times in the sum $\sum_{i=1}^{n-1} a_i( \langle \alpha \rangle)a_{n-i}(\langle \beta \rangle)$, except for the $i \equiv 0 \pmod{4}$ case, which only appears $m-1$ times in the sum. We then see that
$$
 \sum_{i=1}^{n-1} a_i( \langle \alpha \rangle)a_{n-i}(\langle \beta \rangle) = m( 2 \langle 1 \rangle + 2 \langle \alpha \beta \rangle + (t_\alpha+t_\beta) + (\langle \beta \rangle t_\alpha + \langle \alpha \rangle t_\beta)) - (\langle 1 \rangle).
$$
By the previous lemma, $\langle \beta \rangle t_\alpha + \langle \alpha \rangle t_\beta = t_\alpha+t_\beta$, so since $t_\alpha,t_\beta$ are $2$-torsion, this gives us 
$$
a_n(\langle \alpha \rangle + \langle \beta \rangle) = 2m \langle \alpha \beta \rangle + (2m+1) \langle 1 \rangle.
$$
This gives the result for $n \equiv 0 \pmod{4}$.
\end{proof}
\begin{cor}\label{welldefined}
The functions $a_i: \Z[k^\times] \to \widehat{\mathrm{W}}(k)$ factor through the quotient $\Z[k^\times] \to \widehat{\mathrm{W}}(k)$.
\end{cor}
\begin{proof}
By Theorem 4.1 of \cite{La}, we only need to show that for every $a,b \in k^\times$, we have equalities $a_n(\langle \alpha \rangle) = a_n(\langle ab^2 \rangle)$ and $a_n( \langle \alpha \rangle + \langle \beta \rangle) = a_n( \langle a+b \rangle + \langle (a+b)ab \rangle)$. It's clear that the first relation is satisfied. The previous lemma and Lemma $\ref{sumts}$ allow us to see that 
$$
a_i(\langle \alpha \rangle + \langle \beta \rangle) = a_i (\langle (\alpha+\beta) \rangle + \langle (\alpha+\beta)\alpha \beta \rangle),
$$ 
which gives the relation as required.
\end{proof}
This means that we can indeed take $a_n([q])$ for any $[q] \in \widehat{\mathrm{W}}(k)$, which addresses the concerns raised in Remark $\ref{welldefinedconcerns}$. From here we will slightly abuse notation and write $a_n$ to mean the function from $\widehat{\mathrm{W}}(k)=\Z[k^\times]/R \to \widehat{\mathrm{W}}(k)$ as well.
\begin{cor}\label{finalpstruc}
The functions $a_n: \widehat{\mathrm{W}}(k) \to \widehat{\mathrm{W}}(k)$ determine a power structure on $\widehat{\mathrm{W}}(k)$.
\end{cor}
\begin{proof}
We apply Corollary $\ref{powerfns}$, and the result is immediate. 
\end{proof}
\begin{cor}\label{maincor1}
There is a unique power structure, $a_n$, on $\widehat{\mathrm{W}}(k)$ such that:
$$
a_n(\langle \alpha \rangle) = \langle \alpha^n \rangle +  \frac{n(n-1)}2 t_\alpha.
$$
Moreover, if $\mathrm{char}(k)=0$ (resp. $\mathrm{char}(k) \neq 2$) there exists a power structure on $\widehat{\mathrm{W}}(k)$ such that $\chi^{mot}: K_0(\mathrm{Var}_k) \to \widehat{\mathrm{W}}(k)$ (resp. $\mathrm{Tr}: K_0(\Et_k) \to \widehat{\mathrm{W}}(k)$) respects the power structures, then the power structure is the one above.
\end{cor}
\begin{proof}
Uniqueness is immediate, since defining the $a_n$ functions on generators allows us to obtain them for any element. Corollary $\ref{finalpstruc}$ ensures they are well defined, and Corollary $\ref{ansdefn}$ guarantees they satisfy the second condition. Moreover we may write the functions in this form, since $t_\alpha$ is $2$-torsion, so $\frac{n(n-1)}{2} t_\alpha =0$ if $n \equiv 0,1 \pmod{4}$, and $t_\alpha$ if $n \equiv 2,3 \pmod{4}$.
\end{proof}

\subsection{Properties of this power structure}\label{properties}
In this subsection we prove results about the power structure determined by the $a_n$ functions. Since $\mathrm{char}(k) \neq 2$, identify $H^1(k, \Z/2\Z) \cong (k^\times)/(k^\times)^2$. For $\alpha \in k^\times$, write $[\alpha]$ to mean the associated class in $H^1(k, \Z/2\Z)$.

\begin{lemma}\label{cup0}
Let $\alpha,\beta \in (k^\times)/(k^\times)^2 \cong H^1(k, \Z/2\Z)$. Then the cup product $[\alpha] \cup [\beta] = 0$ if and only if there exists a solution to the equation $\alpha x^2 + \beta y^2 = 1$.
\end{lemma}
\begin{proof}
This is a standard fact in Galois cohomology: consider the quaternion algebra $Q(a,b)/k$, which corresponds to the cohomology class $[\alpha] \cup [\beta]$ by page 25 of \cite{SCT}, and so $[\alpha] \cup [\beta] = 0$ if and only if $Q(\alpha,\beta)$ is split. Proposition 1.18 of loc. cit. then tells us that $Q(\alpha ,\beta)$ is split if and only if the conic defined by $\alpha x^2 + \beta y^2 =1$ has a rational point, as required. 
\end{proof}

\begin{lemma}\label{ansnotsym}
For $\alpha \in k^\times$, $a_n(\langle \alpha \rangle) = \langle \alpha^n \rangle$ for all $n$ if and only if $[2] \cup [\alpha] = 0 \in H^2(k, \Z/2\Z).$
\end{lemma}
\begin{proof}
Note that $a_n(\langle \alpha \rangle) = \langle a^n \rangle$ for all $n$ if and only if $t_\alpha=0$, i.e., if and only if
$$
\langle 2 \rangle + \langle \alpha \rangle = \langle 1 \rangle + \langle 2\alpha \rangle \in \widehat{\mathrm{W}}(k).
$$
The left hand side is the isomorphism class of the quadratic form $q$ on a vector space with orthogonal basis elements $v,w$ such that $q( \lambda v + \mu w) = 2\lambda^2 + a\mu^2$. The above equality therefore holds if there exist $x,y, x', y' \in k$ such that $xv + yw$ and $x'v + y'w$ are orthogonal, and
\begin{align*}
2x^2 + ay^2 &= 1\\
2(x')^2 + a(y')^2 &= 2\alpha.
\end{align*}
These $x,y$ exist satisfying the first equality if and only if $[2] \cup [\alpha] = 0 \in H^2(k, \Z/2\Z)$ by the previous lemma. Taking $x' = ay$ and $y' = 2x$ then gives the second equality, so the only obstruction to this isomorphism is the class $[2] \cup [\alpha] \in H^2(k, \Z/2\Z)$.
\end{proof}
\begin{cor}\label{classicalpstruc}
The power structure given by the functions $a_n$ on $\mathrm{\widehat{\mathrm{W}}}(k)$ agrees with the power structure defined by McGarraghy's non-factorial symmetric power, as in Proposition 4.23 of \cite{Mc} if and only if the map $[2] \cup -: H^1(k, \Z/2\Z) \to H^2(k, \Z/2\Z)$ is the zero map. 
\end{cor}
\begin{proof}
Let $S^n$ denote the functions from Definition 4.1 of \cite{Mc} defining the non factorial symmetric power. We see that $S^n(\langle \alpha \rangle) = \langle \alpha^n \rangle$, and a power structure on $\widehat{\mathrm{W}}(k)$ is determined uniquely by its value on elements of the form $\langle \alpha \rangle$. These power structures are therefore equal if for all $\alpha \in k^\times$, $a_n(\langle \alpha \rangle) = \langle \alpha^n \rangle$, so this follows by the previous lemma.
\end{proof}
Note that there is a trivial case for which the map above is always $0$.
\begin{lemma}
If $2 \in (k^\times)^2$, then $[2] \cup -: H^1(k, \Z/2\Z) \to H^2(k, \Z/2\Z)$ is the zero map.
\end{lemma}
\begin{proof}
If $2$ is a square, then $[2]$ is the trivial class in $H^1(k, \Z/2\Z) \cong (k^\times)/(k^\times)^2$, and so the map $[2] \cup -$ is clearly the zero map. 
\end{proof}
In some natural cases, we can explicitly say when this map doesn't vanish.
\begin{lemma}
Let $k$ be a $p$-adic field with uniformiser $\pi$, ring of integers $\mathcal{O}$, and residue field $\F$. Suppose $\mathrm{char}(\F) \neq 2$, and that $2$ is not a square in $\F$. Then $[2] \cup [\pi] \neq 0 \in H^2(k, \Z/2\Z)$. 
\end{lemma}
\begin{proof}
By Lemma $\ref{cup0}$, we only need to show that $2x^2 + \pi y^2 = 1$ has no solutions in $k$. Suppose it does, seeking a contradiction. Note that we cannot have $x,y=0$, as neither $\pi$ or $2$ are squares in $k$. Let $x = u \pi^a$ and $y= v \pi^b$, where $u,v \in \mathcal{O}^\times$. We then see that $\pi^{1-2\alpha} y^2 = \pi^{-2\alpha} - 2u^2$. The valuation of the left hand side is then $1+2b-2\alpha$. If $a \neq 0$, the valuation of the right hand side is $\mathrm{max}(0, -2\alpha)$, which is a contradiction. If $a=0$, we see $1-2u^2$ is a unit, since $2$ is not a square in $\F$. This is a contradiction, which gives the result.
\end{proof}
\begin{lemma}\label{cupprod}
Let $k$ be a number field such that $\sqrt{2} \not\in k$. Then there exists $\alpha \in k^\times$ such that the cup product $[2] \cup [\alpha]$ is non zero in $H^2(k, \Z/2\Z)$.
\end{lemma}
\begin{proof}
By assumption $k(\sqrt{2})/k$ is a proper quadratic extension. By Chebotarev's density theorem, there exist infinitely many primes $\mathcal{O}_k$ which are inert in $k(\sqrt{2})$. Let $\mathfrak{p}$ be one such prime that does not lie above $2$.  Let $\pi$ be a uniformiser for $\mathfrak{p}$ and let $\F$ be the residue field of $k_{\mathfrak{p}}$. Since $\mathfrak{p}$ is inert in $k(\sqrt{2})$, $2$ is not a square in $\F$.  By the previous lemma $[2] \cup [\pi]$ is non zero in $H^2(k_{\mathfrak{p}}, \Z/2\Z)$, so $[2] \cup [\pi] \neq 0 \in H^2(k, \Z/2\Z)$ by naturality of the cup product.
\end{proof}
\begin{cor}\label{2exoticpower}
Let $k$ be a number field such that $\sqrt{2} \not\in k$. Then the power structure on $\widehat{\mathrm{W}}(k)$ defined by the $a_n$ functions does not agree with the power structure on $\widehat{\mathrm{W}}(k)$ defined by either the factorial or non-factorial symmetric powers from \cite{Mc}.
\end{cor}
\begin{proof}
Applying Lemma $\ref{cupprod}$ to Lemma $\ref{ansnotsym}$ gives the result.
\end{proof}
\begin{rem}\label{surprising}
The above corollary may be surprising: an initial guess of the authors was that any power structure on $\widehat{\mathrm{W}}(k)$ that is compatible with the symmetric product on $\widehat{\mathrm{W}}(k)$ would necessarily be some sort of ``symmetric product of quadratic forms", such as either the factorial or non-factorial symmetric power from \cite{Mc}. This is not the case unless the cup product with $[2]$ map $- \cup [2]: H^1(k, \Z/2\Z) \to H^2(k, \Z/2\Z)$ is the zero map, since if $[\alpha] \cup [2] \neq 0$, then there is no quadratic form $q: V \to k$ such that $a_n(\langle \alpha \rangle)=[q] \in \widehat{\mathrm{W}}(k)$.
\end{rem}

\begin{rem}\label{Spheres}
It should be noted that the original definition of $\chi^{mot}$ comes from a construction in motivic homotopy theory (see \cite{Le}). In the setting of motivic homotopy theory, $\chi^{mot}$ takes values in the endomorphism ring of the motivic sphere spectrum and we can identify this ring with $\widehat{\mathrm{W}}(k)$ by Morel's theorem (Theorem 6.4.1 of \cite{Mo}) when $k$ is perfect, and also in the general case by .

 We can see the computational need for the term $t_\alpha$ already in Corollary $\ref{ansdefn}$. A closer study of the identification of $\widehat{\mathrm{W}}(k)$ with the endomorphisms of the motivic sphere spectrum may give a more philosophical reason for the appearance of the term $t_\alpha$, though this is beyond the scope of this paper.
\end{rem}

\section{Compatibility in dimension $0$}
Fix $\mathrm{char}(k) \neq 2$. This section is dedicated to proving the following result.
\begin{thm}
The map $\mathrm{Tr}: K_0(\Et_k) \to \widehat{\mathrm{W}}(k)$ respects the power structures on both sides, where $K_0(\Et_k)$ is given the power structure arising from symmetric powers, and $\widehat{\mathrm{W}}(k)$ is given by the power structure arising from Corollary $\ref{finalpstruc}$.
\end{thm}
In order to give an overview of the proof, we first need some notation.
\begin{defn}
Let $X$ be a variety of dimension $0$ over $k$. Then $X=\Spec(A)$ for $A$ some finite dimensional étale algebra over $k$, so write $[A] \in K_0(\Et_k)$ to mean $[\Spec(A)]$.

Let $Q_k \subseteq K_0(\Et_k)$ be the abelian group generated by the classes $[L]$, where either $L=k$ or $L/k$ is a quadratic extension. In particular $Q_k$ is not a ring. Similarly, let $BQ_k \subseteq K_0(\Et_k)$ be the abelian group generated by classes $[L]$, where $L/k$ is either a biquadratic field extension, a quadratic field extension, or $L=k$. Finally, let $MQ_k\subseteq K_0(\Et_k)$ be the abelian group generated by classes of multiquadratic étale algebras. Note that $MQ_k$ is also the subring generated by $Q_k$.
\end{defn}
This theorem is proven in three parts. We first show directly that the theorem holds for elements of $Q_k$ and $BQ_k$, i.e., if $A \in Q_k$ or $BQ_k$, then $\mathrm{Tr}( [A^{(n)}]) = a_n(\mathrm{Tr}(A))$. We then use twisting results about an equivariant motivic Euler characteristic from chapter $3$ of \cite{YZ} to show that the theorem holds for elements of $MQ_k$. Finally, we appeal to a theorem of \cite{GMS} to give us the result.
\subsection{Quadratic and biquadratic étale algebras}
\begin{lemma}\label{surjquad}
The trace map restricted to $Q_k$ gives us a surjection $\mathrm{Tr}: Q_k \to \widehat{\mathrm{W}}(k)$.
\end{lemma}
\begin{proof}
This is immediate from the proof of Theorem $\ref{surjK0}$.
\end{proof}

\begin{lemma}
For $A \in Q_k$, we have $\mathrm{Tr}(A^{(n)}) = a_n(\mathrm{Tr}(A))$. 
\end{lemma}
\begin{proof}
For $A = k(\sqrt{\alpha})$, this follows by definition of our $a_n$ functions, though we can verify this by using Lemma $\ref{compute}$ and Lemma $\ref{quadratic}$. For general $A$, this follows by Lemma $\ref{subgroup}$.
\end{proof}

We now directly show that the above lemma also holds if we replace $Q_k$ with $BQ_k$. We first compute what the symmetric power of a biquadratic extension is in $K_0(\Et_k)$. We then compute the corresponding $a_n$ of the trace form of the biquadratic extension in $\widehat{\mathrm{W}}(k)$, and then show they agree in Corollary $\ref{bqk}$. This is done by a complicated, though elementary, calculation.
\begin{lemma}
Let $K/k$ be a biquadratic extension with $K = k(\sqrt{\alpha}, \sqrt{\beta})$. Then there is an equality in $K_0(\Et_k)[[t]]$:
\begin{align*}
\sum_{i=0}^{\infty} [K^{(i)}] t^i = \ & (1-t^4)^{-1} [k] \\&+ \frac12\left( (1-t^2)^{-2} - (1-t^4)^{-1} \right) \left([ k(\sqrt{\alpha})]+[k(\sqrt{\beta})] + [k(\sqrt{\alpha\beta})]\right)\\& + \frac12\left( (1-t^4)^{-1} - \frac32 (1-t^2)^{-2} + \frac12 (1-t)^{-4} \right) [K].
\end{align*}
\end{lemma}
\begin{proof}
Let $G = \Gal(K/k) \cong (\Z/2\Z)^2$. Let $X$ be the finite $\Gal_k$ set given by $\Spec(K)(\kbar)$, which is isomorphic to $G$ as a $\Gal_k$ set. The symmetric power $X^{(m)} = \Spec(K)^{(m)}(\kbar)$, so the situation immediately reduces to looking at symmetric powers of finite $\Gal_k$ sets. Moreover, the $\Gal_k$ action on $X$ factors through $G$, and so we can reduce to thinking of the $G$ action on $X$. Note that $X^{(m)}$ will also be a finite $G$ set, so we can study $X^{(m)}$ in the Burnside ring, $B(G)$. Proposition 1.4 of \cite{WebbGSet} allows us to compute a generating series for $X^{(m)}$ in $B(G)$. For $X$ and $G$ as above, this is worked through explicitly on page 5 of loc. cit. to obtain:
\begin{align*}
\sum_{i=0}^{\infty} [X^{(i)}] t^i = \ & (1-t^4)^{-1} [X/G] \\&+ \frac12\left( (1-t^2)^{-2} - (1-t^4)^{-1} \right) \left( [X/G_1]+[X/G_2] + [X/G_3]\right)\\& + \frac12\left( (1-t^4)^{-1} - \frac32 (1-t^2)^{-2} + \frac12 (1-t)^{-4} \right) [X],
\end{align*}
where $G_1, G_2, G_3$ correspond to the three proper and non trivial subgroups of $G$. Converting back from the Burnside ring to $K_0(\Et_k)$ gives the result. 
\end{proof}

\begin{cor}\label{pseries}
Let $K/k$ be as above, and let $q=\mathrm{Tr}(K)$. Then there is an equality in $\widehat{\mathrm{W}}(k)[[t]$:
\begin{align*}
\sum_{i=0}^{\infty} \mathrm{Tr}(K^{(i)}) t^i =\ & (1-t^4)^{-1} \langle 1 \rangle \\&+ \frac12\left( (1-t^2)^{-2} - (1-t^4)^{-1} \right) \left( 2\langle 1 \rangle + \langle 2 \rangle q\right)\\& + \frac12\left( (1-t^4)^{-1} - \frac32 (1-t^2)^{-2} + \frac12 (1-t)^{-4} \right) (q).
\end{align*}
\end{cor}
\begin{proof}
Apply $\mathrm{Tr}$ to the above lemma. Note that 
$$
\mathrm{Tr}\left( [k(\sqrt{\alpha})]+[k(\sqrt{\beta})] + [k(\sqrt{\alpha\beta})]\right) = 3\langle 2\rangle + \langle 2\alpha\rangle+\langle2\beta\rangle+\langle2\alpha\beta\rangle = 2\langle 1 \rangle + \langle 2 \rangle q,
$$
which gives the result as required. 
\end{proof}
\begin{lemma}\label{anbiquadratic}
Let $K/k$ be as above. Let $q = \mathrm{Tr}_{K/k} = \langle 1 \rangle + \langle \alpha \rangle + \langle \beta \rangle + \langle \alpha\beta \rangle$ for some $\alpha, \beta$. We can compute $a_n(q)$ to be:
\begin{equation*}
a_n(q) = \begin{cases}
\frac{1}{4}{ n+3 \choose 3}q & \text{ if $n$ is odd},\\
(\frac{1}{4}( {n+3 \choose 3} - \frac{3n+6}{8})q + \frac{n+2}{4}( 2 \langle 1 \rangle + \langle 2\rangle q)&\text{ if $n \equiv 2 \pmod{4}$},\\
(\frac{1}{4}({n+3 \choose 3} - \frac{3n+2}{8})q + \frac{n}{4}( 2 \langle 1 \rangle + \langle 2\rangle q)+\langle1\rangle&\text{ if $n \equiv 0 \pmod{4}$}.
\end{cases}
\end{equation*}
\end{lemma}

\begin{proof}
We first set out some notation. First, note we have the following easy to verify equality:
$$
t_\alpha + t_\beta + t_{\alpha\beta} = \langle 2 \rangle q - q.
$$
Let $p = \langle 1 \rangle + \langle \alpha \rangle$, so that $q=p+\langle \beta \rangle p$. We can use the axioms of power structures to compute
$$
a_n(q) = \sum_{i=0}^n a_i(p)a_{n-i}(\langle \beta \rangle p).
$$
Suppose first that $n \equiv 1,3 \pmod{4}$. Using Lemma $\ref{compute}$, we can compute $a_i(p)a_{n-i}(\langle \beta \rangle p)$. This is a straightforward computation in all cases. We show the hardest case here and the remaining cases are similar. Consider $i \equiv 1 \pmod{4}$ and $n-i \equiv 2 \pmod{4}$. We compute
\begin{align*}
a_i(p)a_{n-i}(\langle \beta \rangle p) &= \frac{i+1}{2}( \langle 1 \rangle + \langle \alpha \rangle) \cdot \left( \langle 1 \rangle + \frac{n-i}{2}(\langle 1 \rangle + \langle \alpha^2\beta\rangle + t_{\beta} + t_{\alpha\beta}\right)\\
&=\frac{(n-i+1)(i+1)}2 (\langle 1 \rangle + \langle \alpha \rangle) + \frac{i+1}{2} (t_{\beta}+t_{\alpha\beta} + \langle \alpha \rangle(t_\beta + t_{\alpha\beta}))\\
&=\frac{(n-i+1)(i+1)}2 (\langle 1 \rangle + \langle \alpha \rangle),
\end{align*}
where to go to the final line, we use that $t_\gamma$ is $2$-torsion for any $\gamma$, as well as Corollary $\ref{ctab}$. With similar computations,  we can compute
\begin{equation*}
a_i(p)a_{n-i}(\langle \beta \rangle p) = \begin{cases}
\frac{(n-i+1)(i+1)}{2}( \langle \beta \rangle + \langle \alpha \beta \rangle) & \text{ if $i \equiv 0 \pmod{4}$},\\
\frac{(n-i+1)(i+1)}2 (\langle 1 \rangle + \langle \alpha \rangle) &\text{ if $i \equiv 1 \pmod{4}$},\\
\frac{(n-i+1)(i+1)}{2}( \langle \beta \rangle + \langle \alpha \beta \rangle) & \text{ if $i \equiv 2 \pmod{4}$},\\
\frac{(n-i+1)(i+1)}2 (\langle 1 \rangle + \langle \alpha \rangle) &\text{ if $i \equiv 3 \pmod{4}$.}
\end{cases}
\end{equation*}
Summing these from $i=0$ to $i=n$ tells us we must have that for some $m,m'$
$$
a_n(q) =  m(\langle 1 \rangle + \langle \alpha \rangle) + m'(\langle \beta \rangle + \langle \alpha \beta \rangle).
$$ 
Note however, since $a_n(q)$ must be symmetric in the $\langle\alpha\rangle, \langle\beta\rangle, \langle\alpha\beta\rangle$ terms, it must be the case that $m=m'$, and for dimension reasons, we see that $m = \frac14{ n+3 \choose 3}$.

Now suppose that $n \equiv 2 \pmod{4}$. Again using Lemma $\ref{compute}$, we can compute
\begin{equation*}
a_i(p)a_{n-i}(\langle \beta \rangle p) = \begin{cases}
\langle 1 \rangle + \frac{i(n-i)+n}{2}(\langle 1 \rangle + \langle \alpha \rangle) + t_\beta + t_{\alpha\beta} & \text{ if $i \equiv 0 \pmod{4}$},\\
\frac{(n-i+1)(i+1)}2 (\langle \beta \rangle + \langle \alpha\beta\rangle) &\text{ if $i \equiv 1 \pmod{4}$},\\
\langle 1 \rangle + \frac{i(n-i)+n}{2}(\langle 1 \rangle + \langle \alpha \rangle) + t_\alpha & \text{ if $i \equiv 2 \pmod{4}$},\\
\frac{(n-i+1)(i+1)}2 (\langle \beta \rangle + \langle \alpha\beta\rangle) &\text{ if $i \equiv 3 \pmod{4}$.}
\end{cases}
\end{equation*}
Note that $a_n(q)$ is then a sum over these terms, where we run through each term $\frac{n-2}{4}$ times, and then once more over the $i\equiv0,1,2 \pmod{4}$ terms. By symmetry, we see that we have to have the same number of $\langle \alpha \rangle, \langle \beta \rangle$ and $\langle \alpha\beta \rangle$ terms, so we have
$$
a_n(q) = mq + \frac{n+2}{2} \langle 1 \rangle + \frac{n+2}{4}(t_\alpha+t_\beta+t_{\alpha\beta}).
$$
for some integer $m$. We now split this into cases, and verify that this gives us the result.

Note, if $n \equiv 6 \pmod{8}$, then this is simply $mq + \frac{n+2}{2} \langle 1 \rangle$. By counting ranks, we can compute $m= \frac{1}{4}{ (n+3) \choose 3} - \frac{n+2}{8}$.  However, if $n \equiv 6 \pmod{8}$, then $\frac{n+2}{4}$ is even. Since  $2q = 2\langle2\rangle q$, we have.
$$
\left(\frac{1}{4}{n+3 \choose 3} - \frac{3n+6}{8}\right)q + \frac{n+2}{4}\left( 2 \langle 1 \rangle + \langle 2\rangle q\right) = m'q + \frac{n+2}{2}\langle1\rangle,
$$
and again by counting ranks, $m'=m$. 

It remains to show the result for $n \equiv 2 \pmod{8}$. Note that
$$
a_n(q) = mq + \frac{n+2}{2}(\langle 1 \rangle) + \langle 2 \rangle q - q. 
$$
Let $m' = m-1-\frac{n-2}{4}$. Then we see $\frac{n-2}{4}$ is even, so $\frac{n-2}{4}q = \frac{n-2}{4}\langle 2 \rangle q$. Similarly, since $\frac{n+2}{2}$ is even, $\frac{n+2}{2}\langle 1 \rangle = \frac{n+2}{2}\langle 2 \rangle$. We therefore have:
\begin{align*}
a_n(q) &= m'q + \frac{n+2}{2}(\langle 1 \rangle) + \frac{n-2}{4}q +  \langle 2 \rangle q \\
&= m'q + \frac{n+2}{2}(\langle 1 \rangle) + \frac{n+2}{4} \langle 2\rangle q \\
&= m'q + \frac{n+2}{4}( 2 \langle 1 \rangle + \langle 2 \rangle q),
\end{align*}
and by counting ranks, $m' = (\frac{1}{4}{n+3 \choose 3} - \frac{3n+6}{8})$, which gives the result.

Now suppose that $n \equiv 0 \pmod{4}$. Again, we can compute
\begin{equation*}
a_i(p)a_{n-i}(\langle \beta \rangle p) = \begin{cases}
\langle 1 \rangle + \frac{i(n-i)+n}{2}(\langle 1 \rangle + \langle \alpha \rangle)  & \text{ if $i \equiv 0 \pmod{4}$},\\
\frac{(n-i+1)(i+1)}2 (\langle \beta \rangle + \langle \alpha\beta\rangle) &\text{ if $i \equiv 1 \pmod{4}$},\\
\langle 1 \rangle + \frac{i(n-i)+n}{2}(\langle 1 \rangle + \langle \alpha \rangle) + t_\alpha + t_\beta + t_{\alpha\beta} & \text{ if $i \equiv 2 \pmod{4}$},\\
\frac{(n-i+1)(i+1)}2 (\langle \beta \rangle + \langle \alpha\beta\rangle) &\text{ if $i \equiv 3 \pmod{4}$.}
\end{cases}
\end{equation*}
Computing $a_n(q)$ by summing over these terms to obtain
$$
a_n(q) = mq + \frac{n+2}{2}\langle 1 \rangle + \frac{n}{4}(t_\alpha+t_\beta+t_{\alpha\beta}),
$$
where again we use symmetry to deduce that we must have the same number of $(\langle 1 \rangle + \langle \alpha \rangle)$ terms as $(\langle \beta \rangle + \langle \alpha\beta \rangle)$ terms. The $\frac{n+2}{2} \langle 1 \rangle$ term comes by counting the $\langle 1 \rangle$ terms that appear for $i \equiv 0,2 \pmod{4}$. We can compute $m = \frac14{n+3\choose3}-\frac{n+2}{8}$ by taking the rank of both sides.

First, suppose that $n \equiv 0 \pmod{8}$. Then
$$
a_n(q) = mq + \frac{n+4}{2}\langle 1 \rangle
$$
 Note that, since $\frac{n}4$ is even, $\frac{n}4 q = \frac{n}4 \langle 2 \rangle q$, so 
\begin{align*}
\left( \frac14{n+3\choose3}-\frac{n+2}{8}\right)q + \frac{n+4}{2}\langle 1 \rangle &=  \left(\frac14{n+3\choose3}-\frac{3n+2}{8}\right)q + \frac{n}{4}\langle 2 \rangle q + \frac{n+4}{4}\langle 1 \rangle\\
&=  \left(\frac14{n+3\choose3}-\frac{3n+2}{8}\right)q + \frac{n}{4}( 2\langle 1 \rangle + \langle 2 \rangle q) + \langle 1 \rangle,
\end{align*}
as required. Finally, suppose that $n \equiv 4 \pmod{8}$. Then
$$
a_n(q) = mq + \frac{n+2}{2}\langle1\rangle + \langle2\rangle q - q.
$$
Letting $m' = m-1-\frac{n-4}{4}$, and noting that since $\frac{n-4}{4}$ is even, $\frac{n-4}{4}\langle2\rangle q = \frac{n-4}{4} q$, gives us
\begin{align*}
a_n(q) &= m'q + \frac{n+2}{2}\langle1\rangle + \frac{n-4}{4} q + \langle 2 \rangle q \\
&= m'q + \frac{n}{4}(2\langle1\rangle + \langle 2 \rangle q) + \langle 1 \rangle,
\end{align*}
and we see that $m' = \left( \frac14 {n+3 \choose 3} - \frac{3n+2}{8} \right)$, which gives us the result as required.
\end{proof}
\begin{cor}
Let $K/k$ be a biquadratic field extension. Then $\mathrm{Tr}(K^{(n)})= a_n(\mathrm{Tr}(K))$.
\end{cor}
\begin{proof}
Note that $a_n(\mathrm{Tr}(K))$ is computed by the lemma above. We note that we see then that $a_n(\mathrm{Tr}(K))$ is equal to the coefficient of $t^n$ in the power series from Corollary $\ref{pseries}$, which is exactly $\mathrm{Tr}(K^{(n)})$, as required. 
\end{proof}

\begin{cor}\label{bqk}
For any $A \in BQ_k$, we have $\mathrm{Tr}(A^{(n)}) = a_n(\mathrm{Tr}(A))$ for all $A \in BQ_k$. 
\end{cor}
\begin{proof}
As with the quadratic case, this follows by applying Lemma $\ref{subgroup}$ to the above corollary.
\end{proof}

\subsection{Twisting results for the power structure}
In this section, we use results from chapter $3$ of \cite{YZ} about the compatibility of the motivic Euler characterstic with Galois twists in order to prove Corollary $\ref{inductionquad}$, an inductive statement that allows us to deduce that the trace map is compatible with the power structures on all of $MQ_k$. We first recall some results from \cite{YZ}.
\begin{defn}
Let $G$ be a finite Galois module. For $X/k$ a variety with a $G$ action, as in chapter $3$ of \cite{YZ}, we can define $^\sigma \! X$, the twist of $X$ by $\sigma$. Theorem 3.7 of \cite{YZ} says that assignment $X \mapsto \ \! ^\sigma \! X$ gives rise to a ring homomorphism $K_0^G(\mathrm{Var}_k) \to K_0(\mathrm{Var}_k)$. Note also that if $Y$ is a $G$-variety over $k$, then so is $Y^{(n)}$, where we give $Y^{(n)}$ the diagonal action of $G$. 

We have similar twisting results for the Grothendieck--Witt ring as well. For $V/k$ a finite dimensional $G$ representation with a $G$ invariant quadratic form $q$, then we can define the twist, $\ \! ^\sigma\! q$ to be the natural quadratic form on $(V \otimes_k \kbar)^{\Gal_k}$, where $\Gal_k$ acts via $\sigma$ on $V$ and via the canonical action on $\kbar$. Theorem 3.12 of \cite{YZ} then says that the assignment $q \mapsto \ \! ^\sigma \! q$ extends to a ring homomorphism $\widehat{\mathrm{W}}^G(k) \to \widehat{\mathrm{W}}(k)$. Finally, Theorem 3.14 of \cite{YZ} says that these homomorphisms are compatible, ie, $\chi^{mot}(\ \! ^\sigma X) = \ \! ^\sigma \! \chi^{mot}_G(X)$.
\end{defn} 
\begin{rem}
In this subsection, we will state many results in the language of general varieties, rather than varieties of dimension $0$, as certain results will hold in more generality than for simply varieties of dimension zero. However, when working with varieties of higher dimension, we must restrict to characteristic $0$ in order to talk about the power structure on $K_0(\mathrm{Var}_k)$. 
\end{rem}

\begin{lemma}\label{equivariantpstruc}
Suppose $\mathrm{char}(k) = 0$ (resp. $\mathrm{char}(k) \neq 2$). The assignment $[X] \mapsto [X^{(n)}]$ gives rise to a power structure on $K_0^G(\mathrm{Var}_k)$ (resp. $K_0^G(\Et_k)$) that is compatible with the forgetful map.
\end{lemma}
\begin{proof}
This is immediate by the definition of symmetric powers. 
\end{proof}

In this section, we use these twisting results to prove an inductive statement that allows us to use Corollary $\ref{bqk}$ in order to show $\chi^{mot}(X^{(n)}) = a_n(\chi^{mot}(X))$ for a larger class of varieties. This is split into three parts. We first prove some results about twisting certain $G$-varieties by a given cocycle corresponding to a quadratic extension of our base field. We then prove results concerning the compatibility of the power structure $a_n$ with certain $G$ actions, and also compatibility with certain twists. Finally, we pull these results together to obtain the main result of this section, Corollary $\ref{MQRes}$. For this section, fix $K/k$ a quadratic field extension. Let $G= \Gal(K/k) \cong \Z/2\Z$, and let $\sigma: \Gal_k \to G$ be the canonical surjection. 
\begin{defn}

Let $X/k$ be a quasi-projective variety. Give $X \amalg X$ a $G$ action, given by permuting the copies for $X$. We will write $X \amalg^G X$ to mean this variety along with its $G$ action. So we may write $[X \amalg^G X] \in K_0^G(\mathrm{Var}_k)$ and moreover the element $[X \amalg^G X]$ lies over $[X \amalg X]$ under the forgetful map $K_0^G(\mathrm{Var}_k) \to K_0(\mathrm{Var}_k)$. 

Similarly $X \times_k X$ also has a canonical $G$ action, given by permuting the co-ordinates. We write $X \times_k^G X$ to mean the $G$-variety given by this. Write $[X \times_k^G X] \in K_0^G(\mathrm{Var}_k)$, so that $[X \times_k^G X]$ lies over $[X \times_k X]$ under the forgetful map.

Note, if $C = [X]-[Y] \in K_0(\mathrm{Var}_k)$, where $X,Y$ are as above, then we can use the above to also define $C \amalg^G C := [X \amalg^G X] - [Y \amalg^G Y]$. 

Similarly, define $C \times^G_k C := [X \times^G_k X] + [Y \times^G_k Y] - [ (X \times_k Y) \amalg^G (X \times_k Y)]$, so that we can define these operations for any element of $K_0(\mathrm{Var}_k)$. 
\end{defn}

\begin{lemma}\label{symmGX}
Suppose that $\mathrm{char}(k) = 0$ (resp. $\mathrm{char}(k) \neq 2$), and let $X \in K_0(\mathrm{Var}_k)$ (resp. $X \in K_0(\Et_k)$). We can compute $[(X \amalg^G X)^{(n)}] \in K_0^G(\mathrm{Var}_k)$ as follows:
\begin{equation*}
(X \amalg^G X)^{(n)}  = \begin{cases}
\sum_{i=0}^{\frac{n-1}{2}} (X^{(i)} \times_k X^{(n-i)}) \amalg^G ( X^{(i)} \times_k X^{(n-i)}) &\text{ $n$ odd}\\
\sum_{i=0}^{\frac{n-2}{2}} (X^{(i)} \times_k X^{(n-i)}) \amalg^G ( X^{(i)} \times_k X^{(n-i)}) +  \left( X^{(\frac{n}2)} \times_k^G X^{(\frac{n}2)} \right)&\text{ $n$ even}.
\end{cases}
\end{equation*}
\end{lemma}
\begin{proof}
We prove the result for $n$ even, with the case $n$ odd being omitted as it is simpler. We will also only prove the argument in the characteristic $0$ case, with the characteristic $p$ case also following identically.

Since $(-)^{(n)}$ defines a power structure on $K^G_0(\mathrm{Var}_k)$, it is sufficient to prove the above statement for $X$ an actual variety, rather than a class in $K_0(\mathrm{Var}_k)$.

In order to make the $G$ action clear, let $Y \cong X$ so that $X \amalg^G X \cong X \amalg Y$ where the $G$ action swaps $X$ and $Y$. Using the formula for $(X \amalg Y)^{(n)}$, we obtain
$$
(X \amalg Y)^{(n)} = \coprod_{i=0}^n X^{(i)} \times_k Y^{(n-i)},
$$
where the $G$ action sends the copy of $X^{(i)}$ isomorphically onto the copy of $Y^{(i)}$, then so it sends $X^{(i)} \times_k Y^{(n-i)}$ isomorphically onto the copy of $X^{(n-i)}\times Y^{(i)}$.  Therefore for $i \neq \frac{n}{2}$, the variety
$$
(X^{(i)} \times_k Y^{(n-i)}) \amalg (X^{(n-i)} \times_k Y^{(i)})
$$ 
is a closed $G$ invariant subvariety of $(X \amalg Y)^{(n)}$, and there is an isomorphism of $G$ varieties
$$
(X^{(i)} \times_k Y^{(n-i)}) \amalg (X^{(n-i)} \times_k Y^{(i)}) \cong ( X^{(i)} \times_k X^{(n-i)}) \amalg^G (X^{(i)} \times_k X^{(n-i)}).
$$
For $i = \frac{n}{2}$, $X^{(\frac{n}{2})} \times_k Y^{(\frac{n}{2})}$ is $G$ invariant and the $G$ action swaps $X$ and $Y$, so there is an isomorphism of $G$ varieties
$$
X^{(\frac{n}{2})} \times_k Y^{(\frac{n}{2})} \cong X^{(\frac{n}{2})} \times^G_k X^{(\frac{n}{2})}.
$$
Putting this together gives the result as required. 
\end{proof}

\begin{lemma}\label{twistX}
Let $\mathrm{char}(k)=0$ (resp. $\mathrm{char}(k) \neq 2)$ and let $X \in K_0(\mathrm{Var}_k)$ (resp. $K_0(\Et_k)$). Then there is an equality in $K_0(\mathrm{Var}_k)$
$$
\ \! ^\sigma \! (X \amalg^G X) = [X] \cdot [\Spec(K)].
$$ 
\end{lemma}
\begin{proof}
For $X$ the class of a quasi-projective variety (resp. finite étale algebra) this follows immediately by Galois descent, and for $X$ a general class in $K_0(\mathrm{Var}_k)$ (resp. $K_0(\Et_k)$), this follows by the variety (resp. étale algebra) case and Theorem 3.7 of \cite{YZ}.
\end{proof}

\begin{lemma}\label{symmtwist}
Suppose $\mathrm{char}(k) = 0$ (resp. $\mathrm{char}(k) \neq 2)$. Let $X \in K_0(\mathrm{Var}_k)$ (resp. $K_0(\Et_k)$). Then there is an equality in $K_0(\mathrm{Var}_k)$ (resp. $K_0(\Et_k)$)  
$$
\ \! ^\sigma \! ((X\amalg^G X)^{(n)}) \cong ( \ \! ^\sigma \! (X\amalg^GX))^{(n)}.
$$ 
\end{lemma}
\begin{proof}
Since $(-)^{(n)}$ defines a power structure on $K_0^G(\mathrm{Var}_k)$ (resp. $K_0(\Et_k)$), it is sufficient to prove this lemma for $X$ an actual quasi-projective variety (resp. finite étale algebra).
Consider $Y=(X \amalg^G X)^n$. This has a natural action of $G$, via the diagonal action, and a natural action of $S_n$ by permuting the factors. Moreover, these actions commute. We obtain the result since
$$(\ \! ^\sigma \! (X\amalg^GX))^{(n)} = (\ \! ^\sigma\! Y)/S_n= \ \! ^\sigma \! (Y/S_n) = \ \! ^\sigma \! ((X\amalg^GX)^{(n)}).$$ 
\end{proof}

We would like to mimic the above for elements of $\widehat{\mathrm{W}}(k)$. That is, for $q \in \widehat{\mathrm{W}}(k)$, we would like to define $q +^G q$ and $q \times^G q$, elements of $\widehat{\mathrm{W}}^G(k)$, that play the role of $X \amalg^G X$ and $X \times_k^G X$.

\begin{defn}
Let $V$ be a vector space with a quadratic form. There is an canonical $G$ action on $(V \oplus V)$ (resp. $V \otimes V$) that makes the canonical quadratic form $G$ invariant, where the non trivial element of $G$ permutes the copies of $V$. Write $V \oplus^G V$ (resp. $V \otimes^G V$) for these vector spaces with $G$-invariant quadratic forms. Now let $q \in \widehat{\mathrm{W}}(k)$, so that $q = [V] - [W]$ where $V,W$ are quadratic forms over $k$.  Define the following elements of $\widehat{\mathrm{W}}^G(k)$:
\begin{align*}
q +^G q& :=  [V \oplus^G V] - [W \oplus^G W],\\
q \times^G q &:= [V \otimes^G V] + [W \otimes^G W] - [ (V \otimes W) \oplus^G (V \otimes W)].
\end{align*}
Note that $q +^G q$ (resp. $q \times^G q$) is an element of $\widehat{\mathrm{W}}^G(k)$ lying over $q+q$ (resp. $q^2$) under the forgetful map $\widehat{\mathrm{W}}^G(k) \to \widehat{\mathrm{W}}(k)$. 
\end{defn}
\begin{lemma}
For all $q, q' \in \widehat{\mathrm{W}}(k)$, the operations $+^G$ and $\times^G$ satisfy the following:
\begin{enumerate}
\item $(q+q') +^G (q+q') = (q+^Gq) + (q'+^Gq')$.
\item $(q + q') \times^G (q+q') = (q \times^G q) + (q' \times^G q') + (qq' +^G qq')$.
\end{enumerate}
\end{lemma}
\begin{proof}
This follows by expanding $q = [V] - [W]$ and $q' = [V'] - [W']$.
\end{proof}

\begin{lemma}\label{+GxG}
Suppose that $\mathrm{char}(k) = 0$ (resp. $\mathrm{char}(k) \neq 2)$, and let $X \in K_0(\mathrm{Var}_k)$ (resp. $K_0(\Et_k)$) be such that $\chi^{mot}(X) = q \in \widehat{\mathrm{W}}(k)$. Then the $G$ equivariant motivic Euler characteristic $\chi_G^{mot}( X \amalg^G X) = q +^G q$. Similarly, $\chi_G^{mot}( X \times^G X) = q \times^G q$. 
\end{lemma}
\begin{proof}
We will prove the argument for the statement in characteristic $0$, with the characteristic not $2$ case being strictly easier, replacing ``smooth and projective'' with ``the class of a finite étale algebra''.

As in Lemma $\ref{symmGX}$, let $Y$ be a variety such that $X \cong Y$, so $X \amalg^G X = X \amalg Y$ where the $G$ action swaps $X$ and $Y$, and similarly for $X \times^G X$. First, suppose that $X$ is smooth and projective. Let $V = H^{2*}_{dR}(X)$ and $W= H^{2*+1}_{dR}(X)$. As in Lemma 3.4 of \cite{YZ}, we can give $H^{2*+1}_{dR}(X)$ a hyperbolic quadratic form induced by Poincaré duality. 

Note that $X \amalg Y$ is smooth and projective, and $G$ acts on $H^*_{dR}(X \amalg Y) \cong H^*_{dR}(X) \oplus H^*_{dR}(Y)$ by swapping the two isomorphically. Therefore, we see $\chi^{dR}(X \amalg^G Y) = q +^G q$, as required. 

For $X \times^G Y$, the K{\"u}nneth formula gives us $H^*_{dR}(X \times^G Y) \cong H^*_{dR}(X) \otimes H^*_{dR}(Y)$, and the $G$ action swaps the factors. By directly calculating, we can compute
\begin{align*}
\chi^{dR}_G(X \times^G X) &= [H^{2*}_{dR}(X \times^G Y)] - [H^{2*+1}_{dR}(X \times^G Y)] \\&=  [V \otimes^G V] + [W \otimes^G W] - [ (V \otimes W) \oplus^G (V \otimes W)] = q \times^G q.
\end{align*}

For general $X$, Theorem 3.1 of \cite{Bi} tells us that we can write $[X] = [Z]-[U]$, where $Z,U$ are both smooth, projective, varieties. Let $q_Z = \chi^{mot}(Z)$, and $q_U = \chi^{mot}(U)$, so that $q= q_Z -q_U$. We therefore see that
$$
X \amalg^G X = [Z \amalg^G Z] - [U \amalg^G U],
$$ 
so $\chi^{mot}_G(X \amalg^G X) = (q_Z +^G q_Z) - (q_U +^G q_U) = q+^Gq$ as required. Similarly we see
$$
X \times^G X = [Z \times^G Z] + [U \times^G U] - [ (Z \times U) \amalg^G (Z \times U)].
$$ 
Applying $\chi^{mot}_G$ gives $\chi^{mot}_G (X \times^G X) = [q_Z \times^G q_Z] + [q_U \times^G q_U] - [q_Zq_U +^G q_Zq_U]=q \times^G q$, as required. 
\end{proof}

Note that for $q \in \widehat{\mathrm{W}}^G(k)$ there is no canonical way to give a $G$-action to $a_n(q)$ to obtain an element of $\widehat{\mathrm{W}}^G(k)$. However in specific cases we may define a $G$ action.
\begin{defn}
Let $q \in \widehat{\mathrm{W}}(k)$, and consider $a_n(q+q) = \sum_{i=0}^n a_i(q)a_{n-i}(q)$. Writing $a_n(q+q)$ in this way allows us to define a natural element $a_n^G(q+^Gq) \in \widehat{\mathrm{W}}^G(k)$ that lies over $a_n(q+q)$ under the forgetful map. For $n$ odd, this is given by the $G$ action as described above on:
$$
a_n^G(q+^Gq) :=  \sum_{i=0}^{\frac{n-1}{2}} (a_i(q)a_{n-i}(q) +^G a_{n-i}(q)a_i(q))
$$
 For $n$ even, we define
$$
a_n^G(q+^Gq) := \sum_{i=0}^{\frac{n-2}{2}} (a_i(q)a_{n-i}(q) +^G a_{n-i}(q)a_i(q)) +   ( a_{ \frac{n}{2}}(q)\times^G a_{\frac{n}{2}}(q)).
$$
\end{defn}
Note that $a_n^G$ is just notation, not a function. For a general element, $x \in \widehat{\mathrm{W}}^G(k)$, we can't define $a_n^G(x)$, only $a_n^G(q+^Gq)$ for $q \in \widehat{\mathrm{W}}(k)$. In particular, we do not have an analogy of Lemma $\ref{equivariantpstruc}$ for $\widehat{\mathrm{W}}^G(k)$. The $G$ action we prescribe is only possible through the symmetry of $a_n(q+q)$. The reason we have defined this $a_n^G$ are the following two theorems that show that $a_n^G$ is compatible with much of the earlier structure.

\begin{lemma}\label{ang}
Let $\mathrm{char}(k) = 0$ (resp. $\mathrm{char}(k) \neq 2$) $X\in K_0(\mathrm{Var}_k)$ (resp. $K_0(\Et_k)$) be such that $\chi^{mot}(X^{(n)}) = a_n(\chi^{mot}(X)) \in \widehat{\mathrm{W}}(k)$ for all $n$. Then there is an equality in $\widehat{\mathrm{W}}^G(k)$
$$
\chi^{mot}_G\left( (X \amalg^G X)^{(n)}\right) = a_n^G\left( \chi^{mot}(X) +^G \chi^{mot}(X)\right).
$$
\end{lemma}
\begin{proof}
Again, we prove the case when $n$ is even, with the case $n$ odd being easier. Let $q=\chi^{mot}(X)$. Applying $\chi^{mot}_G$ to Lemma $\ref{symmGX}$ gives us
\begin{align*}
\chi^{mot}_G( (X \amalg^G X)^{(n)}) &= \sum_{i=0}^{\frac{n-2}{2}}\chi^{mot}_G\left( [ (X^{(i)} \times_k X^{(n-i)}) \amalg^G ( X^{(i)} \times_k X^{(n-i)})]\right) \\&+\chi^{mot}_G\left( [ X^{(\frac{n}2)} \times_k^G X^{(\frac{n}2)}]\right).
\end{align*}
Lemma $\ref{+GxG}$ computes the right hand side to be
$$
 \sum_{i=0}^{\frac{n-2}{2}} \left[a_i(q)a_{n-i}(q) +^G a_i(q)a_{n-i}(q)\right] + \left[a_{\frac{n}{2}}(q) \times^G a_{\frac{n}{2}}(q)\right],
$$
which is precisely the definition of $a_n^G(q+^Gq)$. 
\end{proof}

\begin{cor}\label{chimotxk}
Let $\mathrm{char}(k) = 0$ (resp. $\mathrm{char}(k) \neq 2$) and let $X \in K_0(\mathrm{Var}_k)$ (resp. $K_0(\Et_k)$) such that $\chi^{mot}(X^{(n)}) = a_n(\chi^{mot}(X))$ for all $n$. Then
$$
\chi^{mot}(X \times_k \Spec(K)) = \ \! ^\sigma \! a_n^G(\chi^{mot}(X) +^G \chi^{mot}(X)).
$$
\end{cor}
\begin{proof}
By Lemma $\ref{twistX}$, we have that $(X \times_k \Spec(K)) = \ \! ^\sigma\!( X \amalg^G X) \in K_0(\mathrm{Var}_k)$.   By Lemma $\ref{symmtwist}$, we obtain $\left(\ \! ^\sigma \! ( X\amalg^G X)\right)^{(n)} = \ \! ^\sigma \! \left( (X \amalg^G X)^{(n)}\right)$, so we have an equality in $K_0(\mathrm{Var}_k)$
$$
(X \times_k \Spec(K))^{(n)} =  \ \! ^\sigma \! \left( (X \amalg^G X)^{(n)}\right).
$$
Applying $\chi^{mot}$ and using Theorem 3.14 of \cite{YZ} gives
$$
\chi^{mot}((X \times_k \Spec(K))^{(n)}) = \ \! ^\sigma \! \chi^{mot}_G(  (X \amalg^G X)^{(n)}).
$$
By Lemma $\ref{ang}$, the right hand side of this is precisely $\ \! ^\sigma \! a_n^G( \chi^{mot}(X) +^G \chi^{mot}(X))$, as required.
\end{proof}

\begin{thm}\label{twistang}
 Let $q \in \widehat{\mathrm{W}}(k)$ and let $\sigma: \Gal_k \to \Z/2\Z$ be as before. Then we can compute
$$
\ \! ^\sigma \! a_n^G(q+^Gq) = a_n( q \cdot \mathrm{Tr}(K)). 
$$
\end{thm}

\begin{proof}
By Lemma $\ref{surjquad}$, $\mathrm{Tr}: Q_k \to \widehat{\mathrm{W}}(k)$ is surjective, so there exists $A \in Q_k$ such that $\mathrm{Tr}(A) = q$.  Since $A \in Q_k$, $A \otimes_k K \in BQ_k$, and $\mathrm{Tr}(A \otimes_k K) = \mathrm{Tr}(A)\mathrm{Tr}(K) = q \cdot \mathrm{Tr}(K)$. 

By Corollary $\ref{bqk}$, we see that $a_n( q \cdot \mathrm{Tr}(K)) =\mathrm{Tr} ((A \otimes_k K)^{(n)})$. However, we can also apply Corollary $\ref{chimotxk}$ to obtain $ \mathrm{Tr}((A \otimes_k K)^{(n)}) = \ \! ^\sigma \! a_n^G(q+^Gq)$, which gives us the result.
\end{proof}

\begin{cor}\label{inductionquad}
Let $\mathrm{char}(k) = 0$ (resp. $\mathrm{char}(k) \neq 2$), and let $X\in K_0(\mathrm{Var}_k)$ (resp. $K_0(\Et_k)$) be such that that $\chi^{mot}( X^{(n)}) = a_n(\chi^{mot}(X))$ for all $n$. Then for any $K/k$, a quadratic field extension,
$$
\chi^{mot}( (X\times_k \Spec(K))^{(n)}) = a_n(\chi^{mot}(X\times_k\Spec(K))).
$$
\end{cor}
\begin{proof}
The induction hypothesis allows us to apply Corollary $\ref{chimotxk}$ to obtain 
$$
\chi^{mot}( (X \times_k \Spec(K))^{(n)}) = \ \! ^\sigma \! a_n^G(\chi^{mot}(X) +^G \chi^{mot}(X)).
$$
Theorem $\ref{twistang}$ then tells us that this is exactly $a_n\left( \chi^{mot}(X) \cdot \chi^{mot}(\Spec(K))\right)$, so the result follows since $\chi^{mot}$ is a ring homomorphism.
\end{proof}
\begin{cor}\label{MQRes}
The map $\mathrm{Tr}: MQ_k \to \widehat{\mathrm{W}}(k)$ satisfies $\chi^{mot}(A^{(n)}) = a_n( \mathrm{Tr}(A))$ for all $A \in MQ_k$ and for all $n$.
\end{cor}
\begin{proof}
We can reduce to the case where $A/k$ is a multiquadratic extension of degree $2^j$, and proceed by induction on $j$. The result holds for $j=2$ by Lemma $\ref{bqk}$, so suppose it holds for $j' < j$. Since $A$ is multiquadratic, write $A = L \otimes K$, where $K/k$ is a quadratic extension, and $L/k$ is a multiquadratic extension with $[L:k] = 2^{j-1}$. By induction
$$
\mathrm{Tr}(L^{(n)}) = \chi^{mot}( \Spec(L)^{(n)}) = a_n\left(\chi^{mot}(\Spec(L))\right) = a_n(\mathrm{Tr}(L)).
$$
Applying Corollary $\ref{inductionquad}$ to $\Spec(L)$ gives us
$$
\mathrm{Tr}(A^{(n)}) = \chi^{mot} ((\Spec(L) \times_k \Spec(K))^{(n)}) = a_n\left(\chi^{mot}( \Spec(L) \times_k \Spec(K))\right) = a_n( \mathrm{Tr}(A)),
$$
which gives the result.
\end{proof}

\subsection{From multiquadratic algebras to all étale algebras}
In this section, fix $k$ to be a field of characteristic not $2$.
\begin{lemma}\label{dimension}
Let $A/k$ be a finite dimensional étale algebra with $\mathrm{dim}_k(A) = m$. Then
$$
\mathrm{rank}(\mathrm{Tr}(A^{(n)})) = {n+m-1 \choose m} = \mathrm{rank}\left(a_n(\mathrm{Tr}(A))\right).
$$
\end{lemma}
\begin{proof}
Note that $\mathrm{rank}(\mathrm{Tr}(A^{(n)}))$ is given by the rank of the trace form on $A^{(n)}$. The trace form is non degenerate, so this is simply $\mathrm{dim}_k(A^{(n)}) ={n+m-1 \choose m}$. The second part follows since
$$
\mathrm{rank}(a_n(\mathrm{Tr}(A))) = \mathrm{rank}( S^n(\mathrm{Tr}(A))),
$$
where $S^n$ is the non-factorial symmetric power of \cite{Mc}, so this follows by Proposition 3.5 of \cite{Mc}. 
\end{proof}

Consider the Witt ring, $\mathrm{W}(k) = \widehat{\mathrm{W}}(k)/ \langle \mathbb{H} \rangle$.  For $q \in \widehat{\mathrm{W}}(k)$, write $\overline{q}$ to mean the image in $\mathrm{W}(k)$. 

\begin{lemma}
For $n$ a non negative integer, the assignment 
\begin{align*}
\{\text{Étale algebras over $k$}\} &\to \mathrm{W}(k)\\
[A] &\mapsto \overline{a_n(\mathrm{Tr}_A) - \mathrm{Tr}_{A^{(n)}}}
\end{align*}
gives rise to a Witt-valued invariant, in the sense of Definition 1.1 of \cite{GMS}.
\end{lemma}
\begin{proof}
The result follows if this is true for both $\overline{a_n(\mathrm{Tr}_A)}$ and $\overline{\mathrm{Tr}_{A^{(n)}}}$, but this is clear.
\end{proof}
\begin{cor}
For any finite dimensional étale algebra over $k$, we have an equality in $\mathrm{W}(k)$:
$$
\overline{a_n(\mathrm{Tr}(A))} = \overline{\mathrm{Tr}(A^{(n)})} \in \mathrm{W}(k).
$$
\end{cor}
\begin{proof}
This statement is equivalent to saying that the Witt-valued invariant from the previous lemma is $0$ for all étale algebras. For multiquadratic étale algebras, this follows by Corollary $\ref{MQRes}$. For $A/k$ a general étale algebra, we apply Theorem 29.1 of \cite{GMS} to get the result.
\end{proof}

\begin{cor}
For any finite dimensional étale algebra over $k$, we have an equality in $\widehat{\mathrm{W}}(k)$:
$$
a_n(\mathrm{Tr}(A)) = \mathrm{Tr}(A^{(n)}).
$$
\end{cor}
\begin{proof}
The previous corollary tells us that this holds in $\mathrm{W}(k)$, so in $\widehat{\mathrm{W}}(k)$, we have 
$$
a_n(\mathrm{Tr}(A)) = \mathrm{Tr}(A^{(n)}) + m\mathbb{H},
$$
for some integer $m$. Applying Lemma $\ref{dimension}$ gives us $m=0$, as required. 
\end{proof}

\begin{cor}\label{maincor2}
Give $K_0(\Et_k)$ the power structure defined by symmetric powers, and give $\widehat{\mathrm{W}}(k)$ the power structure defined by the $a_n$ functions. Then the trace homomorphism
\begin{align*}
K_0(\Et_k) &\to \widehat{\mathrm{W}}(k)\\
[A] &\mapsto [\mathrm{Tr}_A],
\end{align*}
respects the power structures on these rings.
\end{cor}
\begin{proof}
Lemma $\ref{subgroup}$ implies that this is true if and only if $a_n(\mathrm{Tr}(A)) = \mathrm{Tr}(A^{(n)})$, which is the corollary above.
\end{proof}

\end{document}